\documentclass[11pt]{amsart}
\usepackage{amssymb}
\usepackage{graphicx}
\usepackage{xcolor} % A package to add color.
\usepackage{tensor}
\usepackage{fullpage} % Sets all margins to 1 in.  
\usepackage{amsmath}
\usepackage{amsthm}
\usepackage{verbatim}
\usepackage{hyperref}
\usepackage{todonotes, enumitem} 
\usepackage{array} 
\usepackage{enumitem}
\setlist[enumerate]{leftmargin=1.8em}
\setlist[itemize]{leftmargin=1.8em}
\usepackage{seqsplit,mathtools}

\definecolor{green}{rgb}{0,0.8,0} % Redefines the color green.
\newtheorem{theorem}{Theorem}[section]
\newtheorem{corollary}[theorem]{Corollary}
\newtheorem{lemma}[theorem]{Lemma}

\newtheorem{proposition}[theorem]{Proposition}

\theoremstyle{definition}

\theoremstyle{remark}
\newtheorem{remark}[theorem]{Remark}

\numberwithin{equation}{section}
\newcommand{\nrm}[1]{\Vert#1\Vert}

\newcommand{\tld}[1]{\widetilde{#1}}

\newcommand{\nnrm}[1]{{\vert\kern-0.25ex\vert\kern-0.25ex\vert #1 
		\vert\kern-0.25ex\vert\kern-0.25ex\vert}}

\newcommand{\dist}{\mathrm{dist}\,}

\newcommand{\supp}{{\mathrm{supp}}\,}

\newcommand{\lap}{\Delta}

\newcommand{\rd}{\partial}
\newcommand{\nb}{\nabla}

\newcommand{\dlt}{\delta}

\newcommand{\omg}{\omega}

\newcommand{\varep}{\varepsilon}

\newcommand{\bfe}{{\bf e}}

\newcommand{\bbR}{\mathbb R}

\newcommand{\calA}{\mathcal A}

\newcommand{\calF}{\mathcal F}

\newcommand{\calO}{\mathcal O}

\newcommand{\calX}{\mathcal X}

\newcommand{\ackn}[1]{
	\addtocontents{toc}{\protect\setcounter{tocdepth}{1}}
	\subsection*{Acknowledgements} {#1}
	\addtocontents{toc}{\protect\setcounter{tocdepth}{1}} }

\definecolor{purple}{rgb}{0.65, 0, 1}
\definecolor{orange}{rgb}{1,.5,0}

\begin{document}

	\title{Stability of vortex quadrupoles with odd-odd symmetry}
	
	\author{Kyudong Choi}
	\address{Department of Mathematical Sciences, Ulsan National Institute of Science and Technology, 50 UNIST-gil, Eonyang-eup, Ulju-gun, Ulsan 44919, Republic of Korea.}
	\email{kchoi@unist.ac.kr}
	
	\author{In-Jee Jeong}
	\address{Department of Mathematical Sciences and RIM, Seoul National University, 1 Gwanak-ro, Gwanak-gu, Seoul 08826, and School of Mathematics, Korea Institute for Advanced Study, Republic of Korea.}
	\email{injee$ \_ $j@snu.ac.kr}
	
	\author{Yao Yao}
	\address{Department of Mathematics, National University of Singapore, Block S17, 10 Lower Kent Ridge Road, Singapore, 119076, Singapore.}
	\email{yaoyao@nus.edu.sg}
	
	\date{\today}
	
	\maketitle
	
	\renewcommand{\thefootnote}{\fnsymbol{footnote}}
	\footnotetext{\emph{2020 AMS Mathematics Subject Classification:} 35Q35}
	\footnotetext{\emph{Key words: vortex stability, Lamb dipole, variational principle, desingularization problem, point vortex} }
	\renewcommand{\thefootnote}{\arabic{footnote}}

	\begin{abstract} For the 2D incompressible Euler equations, we establish global-in-time ($t \in \mathbb{R}$) stability of vortex quadrupoles satisfying odd symmetry with respect to both axes. Specifically, if the vorticity restricted to a quadrant is signed, sufficiently concentrated and close to its radial rearrangement up to a translation in $L^1$, we prove that it remains so for all times. 
    The main difficulty is that the kinetic energy maximization problem in a quadrant -- the typical approach for establishing vortex stability -- lacks a solution, as the kinetic energy continues to increase when the vorticity escapes to infinity. We overcome this by taking dynamical information into account: finite-time desingularization result is combined with monotonicity of the first moment and a careful analysis of the interaction energies between vortices. The latter is achieved by new pointwise estimates on the Biot--Savart kernel and quantitative stability results for general interaction kernels. Moreover, with a similar strategy we obtain stability of a pair of opposite-signed Lamb dipoles moving away from each other.
	\end{abstract}

	\date{\today}

	\section{Introduction}
	
	The spontaneous creation of large-scale vortex structures is one of the most fascinating features of two-dimensional incompressible fluid motion, clearly demonstrating the tendency of vorticity to concentrate. These structures often appear as nearly radial vortices or symmetric dipoles (a pair of counter-rotating vortices), which suggests a strong form of asymptotic stability. While it is highly challenging to establish the formation of such coherent vortex structures from generic initial data, a more feasible goal is to understand vortex dynamics near these configurations and establish their neutral (Lyapunov) stability.
	
	A typical mathematical framework for construction and  stability for vortex structures, including radial vortex and dipoles  (See Figure~\ref{fig_dipole}(a) for an illustration of the Lamb dipole), is to consider the kinetic energy maximization problem under appropriate constraints, for the planar incompressible Euler equations. In the case of dipoles, one may consider the Euler equations on the upper half-plane $\bbR^{2}_{+}$ and impose a constraint on the impulse (vertical first moment).  
	
	Existing results on dipole stability, despite overcoming significant technical challenges, are often quite restrictive: not only must the vorticity be a small perturbation in $L^{p}$
	type norms of a certain single dipole, but it must also share the same sign as the dipole throughout  $\bbR^{2}_{+}$. These assumptions are crucial within the variational framework, as otherwise the structure of the energy-maximizing set becomes genuinely different. Notably, both assumptions fail in the case of two dipoles with opposite signs, even when they are positioned far apart; see Figure~\ref{fig_dipole}(b) for an illustration.

 	\begin{figure}[htbp]
		\begin{center}
		\hspace*{-1cm}\includegraphics[scale=1]{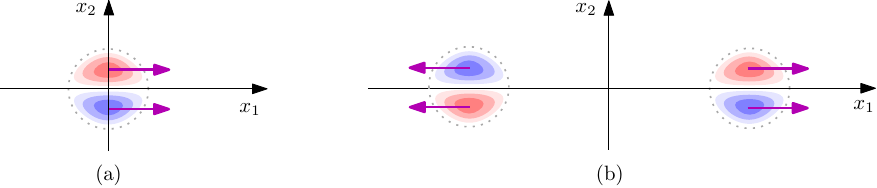}
				\caption{\label{fig_dipole}
		(a) Illustration of a   Lamb dipole, which is a traveling wave solution of the 2D Euler equation.  (b) For a pair of opposite-signed Lamb dipoles far away from each other, there has been no stability result in the literature.}
		\end{center}
		\end{figure}

	This paper takes a step beyond the traditional approach by establishing the dynamical stability of a pair of dipoles moving away from each other, and more generally, for vortex quadrupoles under odd-odd symmetry (odd symmetry with respect to both axes). In such configurations, the problem reduces to considering a single vortex in the first quadrant $Q := (\bbR_{+})^{2}$. However, this domain lacks translation symmetry, and kinetic energy maximizers do not exist, under constraints which are left invariant in time. Our new strategy is to incorporate dynamical information and derive quantitative estimates for possible energy variations of \textit{individual vortices} along the Euler flow under odd-odd symmetry.
		
	\subsection{Main results}
	Let us first recall the two-dimensional incompressible Euler equations in vorticity form, defined on the whole plane $\bbR^{2}$:  \begin{equation}\label{eq:Euler}
		\left\{
		\begin{aligned}
			\rd_t\omg + u\cdot\nb\omg = 0, & \\
			u = \nb^\perp\lap^{-1}\omg &
		\end{aligned}
		\right.
	\end{equation} where $\omg: \bbR\times\bbR^{2}\to\bbR$ and $u: \bbR\times\bbR^{2}\to\bbR^{2}$ represent the vorticity and the velocity, respectively. We shall   need that the first moment is well-defined, and for convenience we introduce the class 
	\begin{equation*}
		\begin{split}
			\calX := \left\{ \omg \in (L^{1} \cap L^{\infty})(\bbR^{2}) \, : \,  
			\int_{\bbR^{2}}|x||\omg(x)| dx < \infty \right\}.
		\end{split}
	\end{equation*} Given $\omg_{0} \in \calX$, the corresponding unique, global-in-time Yudovich solution $\omg(\cdot,t)$ belongs to $\calX$ again for all $t \in \bbR$. Furthermore, we shall consider vorticities with {odd-odd} symmetry \begin{equation*}
		\begin{split}
			\omg(x) = -\omg(\bar{x}) = -\omg(\tld{x}) = \omg(-x) ,\qquad x\in\bbR^{2} 
		\end{split}
	\end{equation*} (where $\bar{x} := (x_1,-x_2)$, $\tld{x} := (-x_1,x_2)$) and the sign condition $\omg(x)\ge0$ in 
	the first quadrant $$Q := \{x\in\mathbb{R}^2\,:\, x_1>0, x_2>0\}.$$

	\medskip 
	
	Our first main result gives forward-in-time orbital stability of a pair of Lamb (or Chaplygin--Lamb) dipoles moving away from each other. We take the Lamb dipole $\omg_{L}$ illustrated in Figure~\ref{fig_dipole}(a), which is an explicit traveling wave solution of \eqref{eq:Euler} satisfying
	\begin{enumerate}
		\item
		$\omg_L(x)>0$ if and only if   $|x|<1$  and $x_2>0$;
		\item odd symmetry with respect to $x_2$-variable, and
		
		\item  unit traveling speed 
		in the sense that 
		$\omega_L(x-t e_1)$ is a solution to \eqref{eq:Euler}.

	\end{enumerate}		
	This special dipole 
	has Lipschitz continuity in $\mathbb{R}^2$ and
	can be obtained as a maximizer of kinetic energy (For more details, we refer to Subsection \ref{intro_lamb}). 
	If we put two opposite-signed Lamb dipoles in the following way:
	$$
	\tilde{\omg}(x_1,x_2)=
	\omg_{L}(x_1-d,x_2)-\omg_{L}(-x_1-d,x_2)\quad \mbox{for}\quad  d\geq 1,
	$$ then 
	$\tilde{\omg}$ becomes a quadrupole under odd-odd symmetry with the sign condition in $Q$, as shown in Figure~\ref{fig_dipole}(b). 
	We obtain the following stability theorem for such a quadrupole  for $d\gg 1$ under odd-odd perturbations: 
	\begin{theorem}\label{thm:main0}
		 
		For each  $\varep>0$, 
		there exist $\dlt=\dlt(\varep)>0$ and $d_{0}=d_{0}(\varep)>0$ such that for all $d\ge d_{0}$, if  
		we consider any odd-odd symmetric 
		initial data $\omg_0\in \mathcal{X} $ with $\omg_0\geq0$ in $Q$ satisfying
		\begin{equation}\label{cond_lamb_prime} 
			\nrm{\omg_0 - \omg_{L}(\cdot_{x}-de_1)}_{L^1\cap L^2(Q)}
			+\nrm{x_2\left(\omg_0 - \omg_{L}(\cdot_{x}-d e_1)\right)}_{L^1(Q)} 
			\leq \dlt,
		\end{equation} 
		then 
		there exists a   function $\tau(\cdot_t):[0,\infty)\to[1,\infty)$ such that 
		the corresponding Euler solution $\omg(\cdot_t)$ 
		satisfies, for any $t\geq0$, \begin{equation*}
			\begin{split}
				\nrm{\omg(t) - \omg_{L}(\cdot_{x}-\tau(t)e_1)}_{L^1\cap L^2(Q)}
				+\nrm{x_2\left(\omg(t) - \omg_{L}(\cdot_{x}-\tau(t)e_1)\right)}_{L^1(Q)} 
				\leq \varep. 
			\end{split}
		\end{equation*} 
  Here, the norm $\|\cdot\|_{L^1\cap L^2}$ is defined by 
  $\|\cdot\|_{L^1}+\|\cdot\|_{L^2}$.
\end{theorem} 
	This result says that if initially two opposite-signed Lamb dipoles are sufficiently far from each other and non-negative in the first quadrant, then each dipole stays close to a Lamb dipole for all positive times, and they would keep moving apart as if they do not see each other. In the statement, restricting $t$ to non-negative values is essential; two  Lamb dipoles ``colliding'' with each other will lead to a serious distortion of the shape, see simulations in \cite{Fu00,Hesthaven} and also in \cite{GF89,Or92,VAvH,KrXu,Gur}.
	 
	The type of stability given in Theorem \ref{thm:main0} is not really specific to the Lamb dipole and carries over to a more general class of dipoles arising as energy maximizers under appropriate constraints, see Remark \ref{rem_general_dip} for the details. 
	
	\medskip

	Our second main result gives global-in-time $(t \in \bbR)$ orbital stability for a concentrated non-negative initial vortex $\rho_0$ in $Q$ that is a small perturbation of $\rho_0^*$ after some translation. Here, $\rho_{0}^*$ is the \emph{radially symmetric decreasing rearrangement} of $\rho_{0}$ (see \cite[Section 3.3]{LL} for its definition and properties).  Namely, we will show that for all time, the vortex stays close to some translation of $\rho_0^*$ in $L^1$ norm, and its center of mass roughly follows a single point vortex trajectory in $Q$. We recall that a single point vortex in $Q$ starting at $p$ (equivalently, four point vortices in $\bbR^{2}$ with odd-odd symmetry) has its trajectory given by $Z_p(t)$, which traverses the curve defined by the equation 
	\begin{equation}\label{eq:orbit-pv}
		\begin{split}
			\calO_{pv}(p) := \left\{ (x_1,x_2) \in Q :  \frac{(x_1x_2)^{2}}{x_1^2 + x_2^2} = \frac{(p_1p_2)^{2}}{p_1^2 + p_2^2}   \right\}
		\end{split}
	\end{equation}
	(\cite{Lamb,Yang}). For $x\in\calO_{pv}(p)$, as $x_1 \to +\infty$, we have $x_2 \to \frac{p_1 p_2}{\sqrt{p_1^2+p_2^2}}$, and similarly $x_1  \to \frac{p_1 p_2}{\sqrt{p_1^2+p_2^2}}$ as $x_2\to+\infty$.  See Figure~\ref{fig_orbit}(a) for an illustration of the curve $\mathcal{O}_{pv}(p)$ and Figure~\ref{fig_orbit}(b) for an illustration of Theorem~\ref{thm:main1}.

		\begin{figure}[htbp]
			\begin{center}
				\hspace*{-1cm}\includegraphics[scale=1]{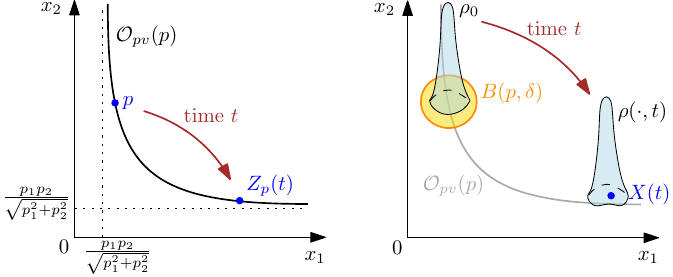}
				
				\emph{(a)  \hspace{5cm} (b)}
				\caption{\label{fig_orbit}
					(a) Illustration of the point vortex trajectory $\mathcal{O}_{pv}(p)$ in \eqref{eq:orbit-pv} and its asymptotic limits; (b) Illustration of Theorem~\ref{thm:main1}. Note that our stability result holds for both positive and negative times. }
			\end{center}
		\end{figure}

	For $p\in\bbR^2$ and $r>0$, we define $B(p,r)$ to be the open ball of radius $r$ centered at $p$. We use $\mathbf{1}_A$ to denote a characteristic function.
 
	\begin{theorem}\label{thm:main1} 
		Consider \eqref{eq:Euler} with odd-odd initial data $\omega_0$. Denote $\rho(\cdot,t) := \omega(\cdot,t)\mathbf{1}_Q$, and assume $\int_{\mathbb{R}^2} \rho_0(x) dx = 1$.
		
		For any $\varep>0, A\geq 1$ and  $p \in Q$, there exists $\delta_0 = \delta_0(\varep, A, p)$, such that for any $\delta \in (0,\delta_0)$, if 
		\begin{equation}
			\label{a_main}
			\supp\rho_0 \subset B(p,\delta),\quad  0\leq \rho_0 \leq A\delta^{-2}, \quad\text{and}  \quad\|\rho_0-\rho_0^*(\cdot-p)\|_{L^1}\leq \delta,
		\end{equation}
		then $\rho(\cdot,t)$ satisfies
		\begin{equation}
			\label{thm_goal1}
			\inf_{a\in \mathbb{R}^2}\|\rho(\cdot,t)-\rho_0^*(\cdot-a)\|_{L^1}\leq \varep \quad\text{ for all }t\in\mathbb{R},
		\end{equation}
		and the center of mass $X(t) := \int_{\mathbb{R}^2} x \rho(x,t)dx$ satisfies
		\begin{equation}
			\label{thm_goal2}
			\dist(X(t), \mathcal{O}_{pv}(p)) \leq \varep \quad\text{for all }t\in\mathbb{R}.
		\end{equation}
			\end{theorem}

	 Combining Theorem~\ref{thm:main1} with the scaling invariance\footnote{Recall if $\omega$ solves \eqref{eq:Euler}, then $\tilde\omega(x,t) = \delta^2\omega(\delta x, \delta^2 t)$ is also a solution.} of \eqref{eq:Euler}, we can rescale the initial data in Theorem~\ref{thm:main1} by a parameter $\delta$, such that both its support size and its $L^\infty$ norm is of order 1 after the rescaling. (Note that the initial support would be centered around the point $p/\delta$ after the rescaling.) This immediately leads to the following result, showing that odd-odd circular patches that starts faraway remain nearly circular (in terms of symmetric difference) for all times. See Figure~\ref{fig_patch} for an illustration.

\begin{corollary}
\label{cor_patch}
Consider \eqref{eq:Euler} with odd-odd initial data $\omega_0$ that is of patch type, with $\omega_0\mathbf{1}_Q = \mathbf{1}_{D_0}$ for some $D_0\subset Q$. Then $\omega(\cdot,t)$ is also of patch type for all $t\in\mathbb{R}$. Denote  the patch in $Q$ by $D(t)$.

For any $\varep > 0, p\in Q$, there exists $\delta_0 = \delta_0(\varep,p)$, such that for any $\delta\in(0,\delta_0)$, if 
\[
D_0 = B(\delta^{-1}p,\pi^{-1/2}),
\]then $D(t)$ satisfies
\[
\inf_{a\in\mathbb{R}^2}|D(t)\triangle B(a,\pi^{-1/2})|\leq \varep \quad\text{for all }t\in\mathbb{R}.
\]
\end{corollary}

		\begin{figure}[htbp]
			\begin{center}
				\hspace*{-1cm}\includegraphics[scale=1]{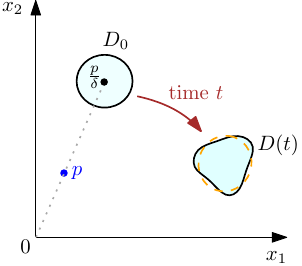}
				
				\caption{\label{fig_patch}
					Corollary~\ref{cor_patch} shows that odd-odd circular patches that starts faraway remain nearly circular  for all times.}
			\end{center}
		\end{figure}
	
	\subsection{Discussion} Our main results provide stability of several types of vortex quadrupoles (in the sense that the vorticity is highly concentrated near four points in the plane) under the odd-odd symmetry. This seems to be the first stability result for vorticities which are not relative equilibria of the Euler equations. More precisely, our first theorem, together with Remark \ref{rem_general_dip}, gives stability for a pair of energy-maximizing dipoles, and the second theorem for odd-odd quadrupoles of radial and monotone decreasing vorticities. These two results are closely related, as energy maximizing dipoles are in general asymptotically radial and monotone in the large impulse limit (see \cite{Burton88,Burton.2021}, \cite[Figure 7]{Leweke.Dizes.Williamson.2016}). The second main result can be also interpreted as a solution to the global desingularization problem for point vortex motion using any radial and monotone decreasing vorticity. To expand upon this point, let us compare our results with recent exciting developments in the desingularization problem under the odd-odd symmetry. 
	 
	\smallskip
	
	\begin{itemize}
		\item D\'{a}vila--del Pino--Musso--Parmeshwar \cite{DdPMP2} established the existence of a solution of \eqref{eq:Euler} that has the form \begin{equation}\label{eq:DdPMP}
			\begin{split}
				\omg(x,t) & = \frac{1}{\varep^{2}}\left[ W\left( \frac{x-ct \bfe_{1} - q \bfe_{2} }{\varep} \right) -  W\left( \frac{x-ct \bfe_{1} + q \bfe_{2} }{\varep} \right) \right. \\
				& \quad \quad \quad  \left. -  W\left( \frac{x + ct \bfe_{1} - q \bfe_{2} }{\varep} \right) +  W\left( \frac{x + ct \bfe_{1} + q \bfe_{2} }{\varep} \right) \right]  + o(1) 
			\end{split}
		\end{equation} for all $t \ge T_{0}$ for sufficiently large $T_{0}$ and for sufficiently small $\varep>0$, with the error term in \eqref{eq:DdPMP} satisfying quantitative decay estimates as $t\to\infty$. The proof is based on the gluing method, which has been very effective in constructing desingularized solutions for the Euler equations (\cite{DdPMW,DdPMW22,DdPMP}). Here, $W$ is a smooth radial vortex given by the solution of a specific elliptic problem, and the analysis in \cite{DdPMP2} builds upon their previous work \cite{DdPMP} which gives a detailed asymptotic information
		of the dipole profile $W$. 
		
		\smallskip 
		  
		\item Hassainia--Hmidi--Roulley \cite{HaHmRo} obtained the existence of time-periodic vortex patches in various bounded domains, and as a particular case, they could establish the existence of time-periodic, localized vortex patch quadruples with odd-odd symmetry, when the fluid domain is given by a disc. Namely, the time-periodic vortex patch traces the orbit of a single point vortex defined in a quarter disc. This is based on applying KAM theory together with the Nash--Moser iteration scheme; earlier remarkable applications of this theory to Euler  were given in \cite{BeHaMa, HaHmMa}.
	\end{itemize}
	
	\smallskip
	
	The results in \cite{DdPMP2,HaHmRo} provide the existence of initial data with odd-odd symmetry whose corresponding solution follows a point vortex orbit. Compared with these works, we do not have a precise control on the error term, except that it remains small in $L^{1}$. However, our method offers several advantages. To begin with, Theorem \ref{thm:main1} works with any profiles of radial and monotone vortex, as long as it is sufficiently concentrated. Similarly, Theorem \ref{thm:main0} could be applied to a more general family of dipoles (see Remark \ref{rem_general_dip}) and potentially to the solution constructed in \cite{DdPMP2} at $t = T_{0}$ as the initial data. Still, the Lamb dipole case is particularly interesting since the vortex profile is genuinely different from typical desingularized objects which are asymptotically radial. In particular, the support of the vortex touches the $x_{1}$-axis. Secondly, while the statement of \cite{DdPMP2} only deals with $t\ge T_{0}$, Theorem \ref{thm:main1} gives global $(t \in \bbR)$ stability, for a similar type of initial data. Lastly, let us point out again that our results do not just give existence of special solutions exhibiting certain behavior, but provide global-in-time stability under natural assumptions on the initial data.

	\subsection{Overview of the proof strategy}

	The basic idea behind our main results is that, if the initial vorticity consists of two dipoles faraway from each other, and each dipole is close to the kinetic energy maximizer under some constraint, then their shapes must remain close to the maximizer for all positive times -- if not, it would lower the kinetic energy. 
	
	We start with a well-known observation that is used in both of our main results. Note that the total kinetic energy $E[\omega(t)]$ is conserved, and in the odd-odd setting, by breaking any non-trivial vorticity into its left and right parts as \[
	\omega(\cdot,t)=\omega(\cdot,t) \mathbf{1}_{\{x_1>0\}} + \omega(\cdot,t) \mathbf{1}_{\{x_1<0\}} =: \omega^r(\cdot,t) - \omega^l(\cdot,t),\] the kinetic energy  can be decomposed into ``dipole energy'' and ``interaction energy'' as follows, where $N(x-y):=-\frac{1}{2\pi}\log|x-y|$: (see Figure~\ref{fig_dipole2} for an illustration)
	\[
	\begin{split}
	E[\omega(t)] &= \iint_{\mathbb{R}^2\times\mathbb{R}^2} \left(\omega^r(x,t) - \omega^l(x,t)\right) \left (\omega^r(y,t) - \omega^l(y,t)\right) N(x-y)  dxdy \\
	&= 2\underbrace{\iint_{\mathbb{R}^2\times\mathbb{R}^2} \omega^r(x,t) \omega^r(y,t)  N(x-y)  dxdy}_{=:E_{\text{dipole}}(t) = E[\omega^r(t)] } \, -\, 2 \underbrace{\iint_{\mathbb{R}^2\times\mathbb{R}^2} \omega^r(x,t) \omega^l(y,t) N(x-y)  dxdy}_{=:E_{\text{inter}}(t)},
		\end{split}
	\]
	where in the last step we used that $E[\omega^l(t)] = E[\omega^r(t)]$. Furthermore, in the odd-odd setting where $\omega_0\geq 0$ in $Q$, it is not difficult to check that $E_{\text{inter}}(t) > 0$ for all $t$. So if $\omega_0$ consists of two dipoles that are sufficiently far away such that $|E_{\text{inter}}(0)|\ll 1$, this leads to 
	\begin{equation}\label{diff_e0}
	E[\omega^r(0)] - E[\omega^r(t)] \ll 1 \quad\text{ for all }t,
	\end{equation}
	where the right hand side can be made arbitrarily small by placing the two dipoles far away initially.
	\begin{figure}[htbp]
		\begin{center}
		\hspace*{-1cm}\includegraphics[scale=0.9]{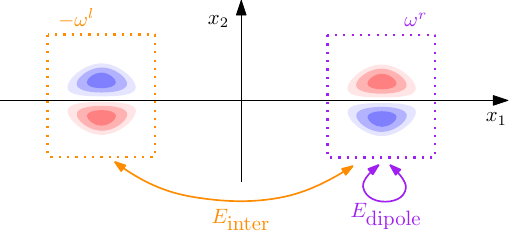}
		
		\caption{\label{fig_dipole2}
		Decomposing the kinetic energy into ``dipole energy'' and ``interaction energy''.  }
		\end{center}
		\end{figure}
	
	Let us choose $\omega^r(\cdot, 0)$ to be the unique maximizer of kinetic energy under some constraint. Heuristically, if there is some \emph{quantitative energy estimate} saying that ``the energy difference controls the shape difference'', and if $\omega^r(\cdot, t)$ also satisfies the same constraint for all $t\geq 0$, then we expect that the above inequality would imply $\omega^r(\cdot, t)$ always stays close to some translation of $\omega^r(\cdot, 0)$.

	Here, one constraint we will impose on the energy maximization problem is the vertical first moment in $Q$. We recall an important work of  Iftimie--Sideris--Gamblin \cite{ISG99}: if 
	$\omg\in\mathcal{X}$ is odd-odd and   $\omg(x)\ge0$ on $Q$,
	then we have monotonicity of the first moment: 
	\begin{lemma}[{Iftimie--Sideris--Gamblin \cite{ISG99}}]
		\label{lemma_center} Define the 
		first moment by \begin{equation*}\label{eq:first-moment}
			\begin{split}
				X(t) =  (X_1(t), X_2(t)), \qquad X_{i}(t) := \int_{Q} x_{i} \omg(x,t) dx .
			\end{split}
		\end{equation*} 
		Then, 
		$X_1(t)$ is increasing in time and $X_2(t)$ is decreasing in time.\footnote{The lemma is contained in the proof of \cite[Theorem 3.1]{ISG99}. Their main theorem gives, for any non-trivial odd-odd symmetric $\omg_{0}\in \calX$ with $\omg_0\geq0$ in $Q$, one has $X_{1}(t) \ge c_{0}t$, where $c_{0}>0$ is a constant depending only on the initial data. While it is remarkable that the statement holds for all such data from $\calX$, it does not give detailed information about the solution. (Still, some very interesting information regarding the infinite time behavior is given in \cite{ILN} under one odd symmetry.) Instead, our main results which will be described below obtain \textit{stability} of   quadrupoles \textit{near energy maximizers}, which provides, as a simple consequence, detailed bounds on $X_{1}(t)$ and $X_{2}(t)$. }
	\end{lemma} 
	
	Below, we will explain how to apply this strategy to both the Lamb dipole setting and the concentrated vortex setting. In each setting, we need to combine rather recent quantitative stability results with several new ingredients. 
	\medskip 
	
	\noindent \textbf{Stability of a pair of Lamb dipoles} (Theorem \ref{thm:main0}). To implement the above strategy, a main ingredient is the orbital stability of (a single pair of) Lamb dipole obtained recently in \cite{AC2019}. Roughly speaking, it says that if an odd symmetric $\omg_{0}$ is non-negative on $\bbR^{2}_{+}$ and is sufficiently close to $\omg_{L}$, then the solution $\omg(\cdot,t)$ corresponding to $\omg_{0}$ again satisfies \begin{equation}\label{eq:AC-Lamb}
		\begin{split}
			\nrm{ \omg(\cdot,t) - \omg_{L}(\cdot - \tau(t)) }_{L^{2}} \ll 1 
		\end{split}
	\end{equation} for all $t\in\bbR$ with some time-dependent shift $\tau(t)$. (The estimate $\tau(t) \simeq W_{L}t$, where $W_L>0$ is the traveling speed, was obtained later in \cite{CJ-Lamb}.) The statement \eqref{eq:AC-Lamb} is a consequence of the fact that $\omg_{L}$ is the unique energy maximizer (actually, a penalized energy \eqref{pen_en}) in the class \begin{equation*}
		\begin{split}
			\calA := \left\{ \omg\in\mathcal{X}\,: \, 
			\mbox{odd symmetric}, \,\, \omg\geq 0 \mbox{ in } \mathbb{R}^2_+, 
			 \,\, \mu[\omg]=\mu[\omg_{L}], \,\, \nrm{\omg}_{L^{2} } =  \nrm{\omg_{L}}_{L^{2} }
			\,\,\right\},		\end{split}
	\end{equation*} where \begin{equation*}
		\begin{split}
			\mu[\omg] := \int_{\bbR^2_+}x_{2}\omg dx. 
		\end{split}
	\end{equation*} While the proof of the orbital stability 
	statement is based on a highly involved contradiction argument, for our application it is essential to turn it into a  \textit{quantitative energy estimate} that says: energy difference controls shape difference from the Lamb dipole. To be more precise, for any constant $\nu\geq \|\omg_L\|_{L^1}$, if $\omg$  belongs to the class $\calA$ and if $\|\omg\|\leq \nu$, then we can prove (see Proposition~\ref{prop:energy-estimate}) \begin{equation}\label{eq:energy-estimate-Lamb}
		\begin{split}
			E[\omg_{L}] - E[\omg] \gtrsim \calF_\nu( \inf_{\tau\in\mathbb{R}}\nrm{\omg_{L} - \omg(\cdot_x-\tau e_1)}_{\calX} ),
		\end{split}
	\end{equation} where we use  an $L^p$ type norm for $\calX$ (see \eqref{defn_x_norm}), and $\calF_\nu$ is 
 a monotone increasing function satisfying $\calF_\nu(0)=0$ and $\calF_\nu(s)>0$ for $s>0$. 
	
	Suppose for simplicity that our (odd-odd symmetric) initial data on the right half plane is given by $\omega_0^r=\omega_L(\cdot-d)$ with $d\gg 1$. 
 By Lemma~\ref{lemma_center}, $\mu[\omega^r(t)]$ is decreasing in time while 
 the initial condition  $d\gg 1$ together with \eqref{diff_e0} guarantees that $\mu[\omg^r(t)],\,t\geq0$ remains sufficiently close to  $\mu[\omg^r_0]=\mu[\omg_L]$. 
 We can then apply a scaled version of  \eqref{eq:energy-estimate-Lamb} to $\omega = \omega^r(t)$ and combine it again with \eqref{diff_e0}, which leads to 
	\[
	\inf_{\tau\in\mathbb{R}}\nrm{\omg_{L} - \omg^r(t, \cdot_x-\tau e_1)}_{\calX}\ll 1 \quad\text{ for all }t\geq 0.
	\]   
	With a  more involved argument, this proof can be adapted to the general case where $\omega^r_0$ is close to (but not exactly equal to) a translation of the Lamb dipole $\omega_L(\cdot-d)$.

	The above argument carries over to general dipoles, once we have a characterization of them in terms of the kinetic energy maximizer. Obtaining dipole stability in this way goes back to the work of Kelvin \cite{kelvin1880}, Arnol'd \cite{Arnold66}, Norbury \cite{Norbury75}, Turkington \cite{Tu83}, Burton \cite{Burton88,Burton96,Burton.2021}, Yang \cite{Yang91} and Burton--Nussenzveig Lopes--Lopes Filho \cite{BNL13}. See also Luzzatto--Fegiz \cite{Luz}. While most of these works deal with dipoles separated from the axis of symmetry, the case of Lamb dipole was settled by Abe--Choi \cite{AC2019}. More recent developments are given by Wang \cite{Wang.2024}. The case of a single concentrated vortex in a disc, which is a similar setup with that of dipoles, is treated by  Cao--Wan--Wang--Zhan \cite{CWWZ}. 
	
	It is a very interesting problem to understand the dipole evolution under the Navier--Stokes flow. A major breakthrough was recently made in Dolce--Gallay \cite{DoGa}, with initial data consisting of two Dirac delta vortices with odd symmetry in $\bbR^{2}$. 
	
	Lastly, it is possible that the methods developed here could give stability of two axisymmetric vortex rings moving away from each other, including the case of Hill's spherical vortex based on \cite{Stability.of.Hill.vortex,CJ-Hill}.

	\medskip 
	
	\noindent \textbf{Stability of concentrated radially decreasing vortex in the quadrant} (Theorem \ref{thm:main1}). In the proof of Theorem~\ref{thm:main1}, we combine the aforementioned ideas from the dipole case with several new ingredients, which we discuss now. As one can see from the statement of Theorem \ref{thm:main1}, the initial data can be concentrated near any point in the quadrant, and therefore it suffices to understand the dynamics in the case $t \ge 0$, thanks to time reversibility of the Euler equations.

 To explain the main ideas of the proof, let us rescale the concentrated initial data to make both its support size and its $L^\infty$ norm to be order 1. (After the rescaling, the initial vortex in $Q$ is centered around $p_0=p/\delta$, which is very far away from the origin.) For simplicity, we consider the special case that the initial data is a vortex patch with area 1 in $Q$ centered at $p_0=p/\delta$. (Note that this is the setting in Corollary~\ref{cor_patch} and illustrated in Figure~\ref{fig_patch}): $\omg_{0}|_{Q} = \mathbf{1}_{D}$ where $D$ is an open set with $|D|=1$. 
 Furthermore, it is conceptually simpler to divide the stability proof into two cases, depending on the location of the center $p_{0} =: (l_{0}, h_{0}) \in Q$: (i) $l_{0} \gg h_{0}$ and (ii) $l_{0} \lesssim h_{0}$. 
	
	The first case is the heart of the proof, and we prove it in Theorem~\ref{thm:to-infty}. Again, our goal is to use \eqref{diff_e0} together with Lemma~\ref{lemma_center} to conclude that $\omega(t)|_Q$ stays close to a disk patch for all $t\geq 0$. Recall that for a \emph{single} disk patch in $\mathbb{R}^2$ (called the \emph{Rankine vortex}), it is well-known to be the energy maximizer among all patches with the same area, and there is a quantitative energy estimate. Namely, denoting by $\bar{\omega}_R$ the disk patch with area 1, \cite{WP85,Tang,Wan86,SV09, CLim22} showed that among all $\omg = \mathbf{1}_{D}$ with $|D|=1$, one has
	\begin{equation}
	\begin{split}\label{stability_patch}
			E[\bar{\omg}_{R}] - E[\omg] \gtrsim \inf_{a\in\mathbb{R}^2}\nrm{ \bar{\omg}_{R}(\cdot-a) - \omg}_{L^{1}}^{2}.
		\end{split}
	\end{equation} 

	However,  one cannot directly apply this stability result to \eqref{diff_e0}, since $\omega_r$ is a dipole, instead of a single vortex with a definite sign. One may be tempted to further decompose $E[\omega_r]$ as the ``self-energy'' and the rest: denoting $\rho := \omega \mathbf{1}_Q$, we have
	\[
	E[\omega_r] = 2\left( E[\rho] +\frac{1}{2\pi}  \iint_{\mathbb{R}^2} \rho(x)\rho(y) \log|x-\bar y|dxdy\right) =: 2 E[\rho] + 2 E_{\text{inter'}}\,,
	\] 
	where $\bar y = (y_1,-y_2)$, and the term $E_{\text{inter'}}$ comes from the interaction between the 1st and 4th quadrant. By \eqref{stability_patch}, $E[\rho]$ achieves its maximum when $\rho$ is a disk patch. Unfortunately, it does not maximize $E_{\text{inter'}}$, which is in fact unbounded from above among patches with the same center of mass and the same area. (To see this, one can stretch a disk patch horizontally while keeping its area and center. As its width being stretched to infinity, $E_{\text{inter'}}$ goes to $+\infty$, whereas $E[\rho]$ goes to $-\infty$.)
	
	To overcome this difficulty, we keep the dipole energy as a single double-integral given by 
	\[
	E[\omega_r] =\frac{1}{\pi} \iint_{\mathbb{R}^2\times\mathbb{R}^2} \rho(x)\rho(y) \log\frac{|x-\bar y|}{|x- y|} dxdy,
	\]
	and take a closer look at the kernel. At $t=0$, since our disk patch is centered at $(l_0,h_0)$ with $h_0\gg 1$, we easily have
	\begin{equation}\label{estimate_0}
	\log \frac{|x-\bar y|}{|x-y|} = -\log|x-y| + \log(2h_0) + O(h_0^{-1}) \quad\text{ for all }x,y \in \supp\rho_0 = B(p_0,1),
	\end{equation}
	In general, this estimate will not hold for future times, since $x,y\in \supp\rho(t)$ may get very far. The key technical ingredient of our argument is the following pointwise estimate on the kernel   (see Lemma \ref{lemma_K} below), where the upper bound resembles \eqref{estimate_0}. Namely, for $h_0\gg 1$, we have 
	\begin{equation}\label{estimate_t}
	\log \frac{|x-\bar y|}{|x-y|} \leq \varphi(|x-y|) + g^{h_0}(2x_2) + O(h_0^{-1}) \quad\text{ for all }x,y \in Q,
	\end{equation}
	where $\varphi(s)=-\log s$ when $s \leq 3$, and $g^{h_0}(s)$ is a concave function that is equal to $\log(2x_2)$ when $s>h_0$. Due to its concavity, as we integrate $g^{h_0}(2x_2)$ against $\rho(x,t)\rho(y,t)dxdy$, using Jensen's inequality with the fact that the vertical center of mass of $\rho$ decreases in time  by Lemma~\ref{lemma_center}, we can bound it from above by $\log(2h_0)$.
	
	Applying the estimates \eqref{estimate_0} and \eqref{estimate_t} to $E[\omega^r(0)]$ and $E[\omega^r(t)]$ respectively, and combining it with \eqref{diff_e0} and the fact that $h_0\gg 1$, we have
	\[
	\iint_{\mathbb{R}^2\times\mathbb{R}^2} \rho_0(x)\rho_0(y) \hspace*{-0.5cm}\underbrace{(-\log|x-y|)}_{=\varphi(|x-y|)  \text{ when }|x-y|\leq 3}  \hspace*{-0.5cm} dxdy - \iint_{\mathbb{R}^2\times\mathbb{R}^2} \rho(x,t)\rho(y,t) \varphi(|x-y|) dxdy \ll 1. 
	\]
	Since both integral contains the same kernel $\varphi(|x-y|)$ (where we used the fact that $\rho_0$ is supported in a disk with radius 1), we can apply a sharp estimate obtained recently in Yan--Yao \cite{YY} (which also holds for general densities, not just patch functions) to conclude that $\rho(\cdot,t)$ stays close to a translation of $\rho_0^*$ for all $t\geq 0$.
	
		To obtain the stability statement in the case (ii), namely when $l_{0} \lesssim h_{0}$, we apply the classical desingularization results of Marchioro--Pulvirenti \cite{MaP83,MaP93}. These results give convergence of evolution of highly concentrated vortex blobs to point vortices in the sense of measures, for a large but finite time interval. Even within a finite time window, we need to upgrade this statement to the following: approximately radial monotone vortices remains so as it moves along the point vortex orbit. (This type of strong desingularization result was already obtained in Davila--del Pino--Musso--Wei \cite{DdPMW} but for a specific radial profile.) For this part we need to apply similar energetic considerations and aforementioned pointwise estimates again.

	\subsection{Organization of the paper} We prove Theorems \ref{thm:main0} and \ref{thm:main1} in Sections \ref{sec:dipole} and \ref{sec:radial}, respectively. These two sections are logically independent of each other. 
	
	\ackn{KC was supported by the National Research Foundation of Korea(NRF) No. 2022R1A4A1032094, RS-2023-00274499. IJ was supported by the National Research Foundation of Korea  grant No. 2022R1C1C1011051, RS-2024-00406821. YY was supported by the NUS startup grant, MOE Tier 1 grant A-0008491-00-00, and the Asian Young Scientist Fellowship. 
	}

	\section{Stability for a pair of Lamb dipoles}\label{sec:dipole}

	\subsection{Notations}
	We recall the notation $\mathcal{X}$:
	\begin{equation*}
		{\mathcal{X}}:=\{
		\omega\in (L^1\cap L^\infty)(\mathbb{R}^2)\,:\,   \int_{\mathbb{R}^2}|x||\omega(x)| \,dx<\infty 
		\}.
	\end{equation*}
	For $\omg\in \mathcal{X}$, we define the stream function $\psi=\psi[\omg]$   by
	$$\psi(x):=[(-\Delta)^{-1}_{\mathbb{R}^2}\omega](x)\,dx=
	-\frac{1}{2\pi}\int_{\mathbb{R}^2}\omega(y)\log{|x-y|}\,dy
	$$ and the kinetic energy $E=E[\omega]$ by
	\begin{equation}\label{def_E}
	E[\omega]:=\int_{\mathbb{R}^2}\psi(x)\omega(x)\,dx=-\frac{1}{2\pi}\int_{\mathbb{R}^2}\int_{\mathbb{R}^2}\omega(x)\omega(y)\log{|x-y|}\,dydx.
	\end{equation}
	In this section, we always assume
	odd    symmetry for vorticities with respect to $x_2$-variable:
	$$ \omg(x_1,-x_2)=-\omg(x_1,x_2).$$
	For convenience, we define the subclasses $ \mathcal{X}_{\text{odd}}, \mathcal{X}_{\text{odd},+}$ of $\mathcal{X}$:
	$$ \mathcal{X}_{\text{odd}}:=\{
	\omega\in \mathcal{X}\,:\, \mbox{$\omega$ is odd symmetric with respect to $x_2$-variable}
	\}$$ and
	$$ \mathcal{X}_{\text{odd},+}:=\{
	\omega\in \mathcal{X}_{\text{odd}}\,:\, \omg\geq0 \mbox{ for } x_2>0
	\}.$$
	For   $\omega\in \mathcal{X}_{\text{odd}}$,  its kinetic energy is obtained by
	\begin{equation*}\label{defn_en_half}
		\begin{split}
			E[\omg] =2 \iint_{\bbR^2_+ \times \bbR^2_+} \omg(x) G_{\bbR^2_+}(x,y) \omg(y) dxdy,
		\end{split}
	\end{equation*}  where
	$$
	G_{\bbR^2_+}(x,y) := \frac{1}{4\pi} \log\left( 1 + \frac{4 x_2 y_2}{|x-y|^{2}} \right)>0\quad\mbox{for}\quad x\neq y. 
	$$
	We observe
	$E[\omega]\ge 0$ for any $\omg\in  \mathcal{X}_{\text{odd},+}$. For future use,  we borrow the following estimate:
	\begin{lemma}[Proposition 2.1 of \cite{AC2019}]\label{lem_est_en_diff}
		For any $\omg,\tilde\omg\in \mathcal{X}_{\textup{odd}}$, we have
		\begin{equation}\label{est_en_diff}
			|E[\omg]-E[\tilde\omg]|\lesssim
			\nrm{\omg-\tilde\omg}_{L^1(\mathbb{R}^2)}^{1/2}\cdot
			\nrm{\omg-\tilde\omg}_{L^2(\mathbb{R}^2)}^{1/2}\cdot
			\nrm{x_2(\omg+\tilde\omg)}_{L^1(\mathbb{R}^2)}^{1/2}\cdot
			\nrm{\omg+\tilde\omg}_{L^1(\mathbb{R}^2)}^{1/2}.
		\end{equation}
		
	\end{lemma}
	\subsection{Lamb dipole (Chaplygin--Lamb dipole)}\label{intro_lamb}
	
	In this subsection, we consider 	the Lamb dipole {(or Chaplygin--Lamb dipole)  $\omega_L\in \mathcal{X}_{\text{odd},+}$ introduced  by 
		H. Lamb \cite[p231]{Lamb} in 1906 and,  independently, by S. A. Chaplygin in 1903 \cite{Chap1903}, \cite{Chap07} 
		(also see \cite{MV94}).} It is simply defined by 
	\begin{equation}   \label{lamb}
		\omega_L(x):=\begin{cases} 
			g(r)\sin\theta,\quad & r\leq 1,\\
			0 ,\quad & r>1,\end{cases}
	\end{equation} in polar coordinates $x_1=r\cos\theta$ and $x_2=r\sin\theta$, where $g$ is defined by
	\begin{equation*} 
		g(r):=\left(\frac{-2c_L}{    J_0(c_L)}\right)J_1(c_L r),\quad 0\leq r\leq 1.
	\end{equation*} 
	Here, $J_{m}(r)$ is the $m$-th order Bessel function of the first kind, the constant $c_L=3.8317\dots>0$ is the first (positive) zero point of $J_1$, and  $J_0(c_L)<0$.
	We note
	$g\in C^2([0,\infty))$ and satisfies
	$$g(0)=g(1)=0, 
	\quad  \mbox{and} \quad g(r)>0,\quad r\in(0,1).$$  
	The Lamb dipole $\omega_L\in \mathcal{X}_{\text{odd},+}$ satisfies  the following properties:
	\begin{enumerate}[label=(\alph*)]
		\item $\omg_L$ is Lipschitz continuous in $\mathbb{R}^2$.
	 
		\item The support of $\omg_L$ is the unit disk. More precisely, 
		$\omega(x)\neq 0$ if and only if $|x|<1$ and $x_2\neq 0$.
		 
		\item Its \emph{impulse} is $\mu_L:=\int_{\mathbb{R}^2_+}x_2\omega_L(x) \,dx=\pi.$
		\item It has unit traveling speed $W_L=1$ in the sense that $\omega_L(x-t e_1)$ is a solution for the 2D Euler equations.
		Indeed, it satisfies
		\begin{equation}\label{iden_psi_omg}
			\omega_{L}=(c_L)^2 f_L(\psi_L-  x_2)\quad\mbox{in}\quad\mathbb{R}^2_+:=\{x\in\mathbb{R}^2\,:\, x_2> 0\},  
		\end{equation} where $f_L$ is simply 
		\begin{equation*} 
			f_L(s)=s^+=\begin{cases} s,\quad &s>0,\\ 0,\quad &s\leq 0,\end{cases}
		\end{equation*}  
		and the stream function $\psi_L$ is defined by $\psi[\omg_L]$.
		\item  \label{property5}
		Let us denote $\kappa_L:=\int_{\mathbb{R}^2_+}(\omega_L)^2 \,dx\in(0,\infty)$.
		Then $\omega_L$ is a maximizer of the kinetic energy 
		$
		E[\omega] 
		$
		in the class 
		\begin{equation}\label{defn_cla_A}
			\mathcal{A}:=\left\{
			\omega\in \mathcal{X}_{\text{odd},+}:\, 
			\, \int_{\mathbb{R}^2_+}x_2\omega(x) \,dx=\pi\,\,\mbox{and}\,
			\int_{\mathbb{R}^2_+}\omega^2(x) \,dx=\kappa_L
			\right\}.
		\end{equation}
		
		\item Every other maximizer of kinetic energy in $\mathcal{A}$ is a $x_1$-translation of $\omega_L$.\\
	\end{enumerate}
	The velocity $u_L$ also can be written explicitly due to the relation
	$
	u_L=-\nabla^{\perp}\psi_L$ (e.g. see \cite{CJ-Lamb}).
	\begin{remark}
		All the properties above can be easily found in literature including \cite{Lamb, AC2019} except \ref{property5}. However, it is easy to deduce the property \ref{property5}   from  the result of \cite{AC2019}. Indeed,    \cite[Theorem 1.5]{AC2019} says, after adjusting to our notation, that $\omega_L$ is the unique maximizer (up to $x_1$-translations)  of the \textit{penalized} energy 
		\begin{equation}\label{pen_en}
			\hat{E}[\omg]:=E[\omega]-\frac{2}{(c_L)^2}\int_{\mathbb{R}^2_+} \omega^2\,dx
		\end{equation}
		in the following (larger) admissible class
		$$\hat{K}:=\left\{\omg\in \mathcal{X}_{\text{odd},+}:\,
		\int_{\mathbb{R}^2_+}x_2\omega(x) \,dx=\pi \right\}.
		$$ Since $\omega_L\in\mathcal{A}\subset \hat{K}$, 
		the property \ref{property5} is obtained.
	\end{remark}
	
	\medskip
	
	By scaling, we have a two-parameter family of Lamb dipoles: 
	\begin{equation}\label{defn_lamb_ab}
		\omega_L^{a,b}(x)=b\cdot\omega_L\left(\frac{x}{a}\right),\quad a>0,\quad b>0.
	\end{equation}
	Each dipole $\omega_L^{a,b}$ is a maximizer of the kinetic energy in the following admissible class 
	\begin{equation*}
		\mathcal{A}_{\mu,\kappa}:=\left\{
		\omega\in \mathcal{X}_{\text{odd},+}:\, 
		\, \int_{\mathbb{R}^2_+}x_2\omega(x) \,dx=\mu\mbox{ and }
		\int_{\mathbb{R}^2_+}\omega^2(x) \,dx=\kappa
		\right\}
	\end{equation*} (e.g. $\mathcal{A}=\mathcal{A}_{\pi,\kappa_L}$)  for the choice $$\mu(a,b)=(a^3\cdot b)\pi\quad\mbox{and}\quad \kappa(a,b)=(a\cdot b)^2\kappa_L,$$
	or equivalently
	\begin{equation}\label{formula_ab} 
		a(\mu,\kappa)=(\mu/\pi)^{1/2}(\kappa_L/\kappa)^{1/4}\quad\mbox{and}\quad
		b(\mu,\kappa)=(\pi/\mu)^{1/2}(\kappa/\kappa_L)^{3/4}.
	\end{equation} Then we get the traveling speed
	$W_L^{a,b}=ab W_L=ab=(\kappa/\kappa_L)^{1/2}$ and the energy
	\begin{equation}\label{iden_en_scaling} 
		E[\omega_L^{a,b}]=a^4\cdot b^2 \cdot E[\omega_L]=(\mu/\pi)\cdot(\kappa/\kappa_L)^{1/2}\cdot E[\omg_L].
	\end{equation} We also note
	\begin{equation}\label{conv_L1}
		\omg_{L}^{a,b}\to\omg_{L} \quad\mbox{in} \quad {L^1(\mathbb{R}^2)}\quad\mbox{as}\quad (a,b)\to (1,1)\quad\mbox{or equivalently}\quad (\mu,\kappa)\to(\pi,\kappa_L).
	\end{equation}

	\begin{remark}
		The relation \eqref{iden_en_scaling} gives the following consequence:\\  The dipole $\omega_L$ is the maximizer of the kinetic energy up to $x_1$-translations in the   admissible class\begin{equation*} 
			\left\{
			\omega\in \mathcal{X}_{\text{odd},+}:\, 
			\, \int_{\mathbb{R}^2_+}x_2\omega(x) \,dx\leq \pi\,\,\mbox{and}\,
			\int_{\mathbb{R}^2_+}\omega^2(x) \,dx\leq \kappa_L
			\right\},
		\end{equation*} 
		which is larger than the class $\mathcal{A}$ in \eqref{defn_cla_A}.\end{remark} 
	\begin{remark} The constant $\kappa_L$ can be explicitly obtained:
		\begin{equation}\label{iden_kappa} 
			\kappa_L=\pi\cdot (c_L)^2.
		\end{equation}
		Indeed, when 
		$\mu=\pi$ with general $\kappa\in(0,\infty)$, which means
		$a=(\kappa_L/\kappa)^{1/4}$ and
		$b=(\kappa/\kappa_L)^{3/4}$, we have, by the fact that $\omega_L$ minimizes \eqref{pen_en},
		$$\hat{E}[\omg_L^{a,b}]\leq \hat{E}[\omg_L]\quad\mbox{for any}\quad \kappa\in(0,\infty).$$ By \eqref{iden_en_scaling}, we get
		$$
		\sqrt{\frac{\kappa}{\kappa_L}}\cdot E[\omg_L]-\frac{2\kappa}{(c_L)^2}\leq E[\omg_L]-\frac{2\kappa_L}{(c_L)^2}\quad\mbox{for any}\quad \kappa\in(0,\infty).
		$$ It implies
		$$E[\omg_L]=\frac{4\kappa_L}{(c_L)^2}.$$ On the other hand, from \eqref{iden_psi_omg},
		\begin{equation*}\begin{split}
				\frac{1}{2}E[\omg_L] &=\int_{\bbR^2_+}\psi_L(x)\omg_L(x)dx
				=\int_{\bbR^2_+}(\psi_L(x)-x_2)\omg_L(x)dx+\int_{\bbR^2_+}x_2\omg_L(x)dx\\
				&=\int_{\bbR^2} \frac{\omg_L(x)}{(c_L)^2}\cdot\omg_L(x)dx+\pi=\frac{\kappa_L}{(c_L)^2}+\pi,
			\end{split}
		\end{equation*} which gives \eqref{iden_kappa}.
		 
	\end{remark}
	\subsection{Stability on the Half-plane} 
	For $\omega\in \mathcal{X}$, we use the following norm:
 \begin{equation}\label{defn_x_norm}
	\|\omega\|_\mathcal{X}:=\int_{\mathbb{R}^2}|x_2||\omega(x)|\,dx+\sqrt{\int_{\mathbb{R}^2}|\omega(x)|^2\,dx}.   
 \end{equation}
	First, we borrow the following compactness statement from \cite{AC2019}: \begin{proposition}\label{prop_cpt_lamb}
		Let $\{ \omg_{n} \}\subset \mathcal{A}$ defined in \eqref{defn_cla_A} be a sequence of vorticities 	 
		such that \begin{equation*}
			\begin{split}
				\sup_n\|\omega_n\|_{L^1}<\infty\quad\mbox{and}\quad	 
				E[\omg_{n}] \to E[\omg_{L}]\quad \mbox{as}\quad n\to\infty. 
			\end{split}
		\end{equation*} Then 
		there exists a subsequence $\{\omg_{n_{k}}\}_{k=1}^\infty$ and a sequence
		$\{\tau_k\}_{k=1}^\infty\subset \mathbb{R}$ such that  
		$$\|\omg_{n_{k}}(\cdot + \tau_{k} e_{1}) - \omg_{L}\|_\mathcal{X}\to 0\quad\mbox{as}\quad k \to \infty.$$  
	\end{proposition}
	Here $\omg_L$ is the Lamb dipole defined in \eqref{lamb}. This compactness result is  contained in    \cite[Theorems 1.3 and 1.5]{AC2019}. We now show that such a compactness statement implies an \textbf{energy estimate}:
	\begin{proposition}[Energy estimate]\label{prop:energy-estimate}  
		Let $\nu>0$ be any constant satisfying $\nu\geq\|\omg_L\|_{L^1}$. Then there exists  a 
  function $F_{\nu} : [0,\infty) \to [0,\infty)$ satisfying \begin{itemize}
			\item $F_{\nu}$ is   monotonically increasing,
			\item $F_{\nu}(0)=0$ and $F_{\nu}(s)>0$ for any $s>0$,
			\item {Energy estimate:}  for all $\omg\in \mathcal{A}$ satisfying $\|\omg\|_{L^1}\leq\nu$, we have
			\begin{equation*}\label{est_en}
				\begin{split}
					E[\omg_{L}] - E[\omg] \ge F_{\nu}( \inf_{\tau\in\mathbb{R}}\nrm{\omg_{L} - \omg(\cdot +\tau  e_1)}_{\mathcal{X}} ).
				\end{split}
			\end{equation*} 
		\end{itemize}
	\end{proposition}
	
	\begin{proof} 
		Let $\nu\geq\|\omg_L\|_{L^1}$.
		First we note that for $\omega \in \mathcal{A}$,
		$$
		\inf_{\tau\in\mathbb{R}}\nrm{\omg_{L} - \omg(\cdot +\tau  e_1)}_{\mathcal{X}} \leq
		\nrm{\omg_{L}}_{\mathcal{X}}+\nrm{\omg}_{\mathcal{X}}<\infty. 
		$$
		For $s\ge 0$, we denote the subclass   \begin{equation*}
			\begin{split}
				\mathcal{S}^{s,\nu} = \left\{ \omg \in \mathcal{A}: \inf_{\tau\in\mathbb{R}}\nrm{\omg_{L} - \omg(\cdot +\tau  e_1)}_{\mathcal{X}} \ge s\quad\mbox{and} 
				\quad \|\omg\|_{L^1}\leq\nu\right\}. 
			\end{split}
		\end{equation*} Then, we define \begin{equation}\label{def_F_nu}
			\begin{split}
				F_{\nu}(s) := \inf_{ \omg \in \mathcal{S}^{s,\nu} } \left( 	E[\omg_{L}] - E[\omg] \right) \ge 0. 
			\end{split}
		\end{equation} 
  Clearly, 
		$F_{\nu}(s)\leq E[\omg_L]<\infty$, 	
		$F_{\nu}(0)=0$ (due to $\omg_L\in \mathcal{S}^{0,\nu}$) and $F_{\nu}(\cdot_s)$ is monotonically increasing in $s\geq 0$, since $	\mathcal{S}^{s,\nu}$ is a smaller class of functions as $s$ increases.\\
		
		Now we show $F_{\nu}(s)>0$ for all $s>0$. 	
		Assume that there is some $s_{1}>0$  
		satisfying  $F_{\nu}(s_{1})=0$. This means that there is a sequence $\{ \omg_{n} \}$, belonging to the class $ \mathcal{S}^{s_1,\nu}\subset \mathcal{A}$, which verifies $E[\omg_{n}]\to E[\omg_{L}]$. However, thanks to the compactness (Proposition \ref{prop_cpt_lamb}),
		the energy convergence implies that there is a subsequence which converges strongly in $\mathcal{X}$ to $\omg_{L}$ after proper translations. This is a contradiction to $\inf_{\tau\in\mathbb{R}}\nrm{\omg_{L} - \omg(\cdot +\tau  e_1)}_{\mathcal{X}} \ge s_1>0$.
	\end{proof}
 
	\begin{remark}
		By scaling as in \eqref{defn_lamb_ab}, we have the following energy estimate 
		for any $\nu\geq\|\omg_L\|_{L^1}$ and 
		for any $\mu,\kappa\in(0,\infty)$: For any
		$\omega\in \mathcal{A}_{\mu,\kappa}$ with 
		$\|\omg\|_{L^1}\leq\sqrt{\mu/\pi}(\kappa/\kappa_L)^{1/4}\nu$,
		we have
		\begin{equation}\label{est_en_gen}
			\begin{split}
				E[\omg_{L}^{a,b}] - E[\omg] \ge 
				(\mu/\pi)\cdot\sqrt{(\kappa/\kappa_L)}		\cdot
				F_{\nu}\left(\left( \max{\{(\mu/\pi),\sqrt{(\kappa/\kappa_L)}\}	}\right)^{-1}\cdot \inf_{\tau\in\mathbb{R}}\nrm{\omg_{L}^{a,b}- \omg(\cdot +\tau  e_1)}_{\mathcal{X}} \right).
			\end{split}
		\end{equation} where
		$\omg_L^{a,b}$ is defined in 
		\eqref{defn_lamb_ab} with
		$a=a(\mu,\kappa), b=b(\mu,\kappa)$  defined in \eqref{formula_ab}.
		
	\end{remark}
	\subsection{Orbital stability on the first quadrant}
	
	For the dynamics of non-negative vorticities on the first quadrant 
	$Q$, 
	let us consider  \textit{nonnegative} vorticity $\rho$ in $Q$: $$\rho \in L^{1}\cap L^{\infty}(Q)\quad\mbox{satisfying}\quad \int_Q |x|\rho(x)\, dx<\infty.$$

	We denote ${\omg}$ the odd-odd extension of $\rho$ into $\mathbb{R}^2$. We note that the extension $ {\omg}$ consists of two opposite dipoles $\omega^{{r}}, \omega^{{l}}$
	in $\mathbb{R}^2$ in the sense that 
	\begin{equation}\label{defn_w1}
		{\omg}=\omega^{{r}}-\omega^{{l}}\quad\mbox{in}\quad\mathbb{R}^2,\quad\omega^{{r}}:= {\omg}\mathbf{1}_{\{x_1>0\}}\in \mathcal{X}_{\text{odd},+},\quad\mbox{and}\quad
		\omega^{{l}}:=- {\omg}\mathbf{1}_{\{x_1<0\}}\in \mathcal{X}_{\text{odd},+}.
	\end{equation}
	We observe the first dipole $\omg^{{r}}$ is supported in $\{x\in\mathbb{R}^2\,:\,x_1>0\}$, and the second one $\omg^{{l}}$ is the mirror image of the first one $\omg^{{r}}$ with respect to $x_2$-axis:
	$$
	\omg^{{l}}(x_1,x_2)=\omg^{{r}}(-x_1,x_2) \quad\mbox{for}\quad x=(x_1,x_2)\in\mathbb{R}^2.
	$$
	
	The conserved kinetic energy consists of two parts: \begin{equation*}
		\begin{split} 
			 E[ {\omg}]=  E[\omg^{{r}}]+E[\omg^{{l}}] -2E_{\text{inter}}[\omg]= 2(E[\omg^{{r}}] - E_{\text{inter}}[\omg]),
		\end{split}
	\end{equation*} where 
	the  interaction energy is defined by 
	\begin{equation*}
		\begin{split}
			E_{\text{inter}}[\omg] :&=-\frac{1}{2\pi}
			\iint_{\bbR^2 \times \bbR^2} \omg^{{r}}(x ) \big(\log{|x-y|}\big) \omg^{{l}}(y ) dxdy 	\\
			&=2\iint_{\bbR^2_+ \times \bbR^2_+} \omg^{{r}}(x) G_{\bbR^2_+}(x,y) \omg^{{l}}(y) dxdy \\
			&=2
			\int_{Q'}\int_{Q}  
			\rho(x_1,x_2) G_{\bbR^2_+}(x,y) \rho(-y_1,y_2) dxdy 	 >0\quad\mbox{when}\quad \rho\geq 0 \quad
\mbox{is non-trivial,} 
		\end{split}
	\end{equation*} 
	where $Q'$ is the second quadrant $\{x\in\mathbb{R}^2\,:\, x_1<0, x_2>0\}$.
	This \textit{positive} quantity $E_{\text{inter}}[\omg]$ is  the \textit{negative} contribution in energy of  $ \omg$ from interaction between two dipoles $\omg^{{r}}$ and $(-\omg^{{l}})$. \\

	We easily observe that if  two dipoles are far from each other, then 
	the interaction energy is negligible. More precisely, assuming that the  support of $\rho$ lies in  $\{x\in Q\,:\, x_1\geq d\}$ for   $d>0$,
	we have a simple estimate \begin{equation}\label{est_int_en}
		\begin{split}
			E_{\text{inter}}[\omg] & \le C\int_{Q'}\int_{Q}  \frac{x_2 y_2}{|x-y|^{2}} \rho(x_1,x_2)  \rho(-y_1,y_2) dxdy \\ & \le \frac{C}{4d^{2}}\int_{Q'}\int_{Q} x_2 y_2  \rho(x_1,x_2)  \rho(-y_1,y_2) dxdy = \frac{C\mu^{2}}{4d^{2}}
		\end{split}
	\end{equation} where \begin{equation*}
		\begin{split}
			\mu := \int_{Q} x_2\rho(x)dx=\frac{1}{2}\int_{\mathbb{R}^2} x_2\omg^{{r}}(x)dx=\frac{1}{4}\int_{\mathbb{R}^2} |x_2\omg(x)|dx\geq 0. 
		\end{split}
	\end{equation*}

	Recall that  $\omg_L$ is the (solo) Lamb dipole in \eqref{lamb}.  We call $\omg_{L,d}$  
	a pair of  two opposite Lamb dipoles, given by
	$$\omg_{L,d}(x_1,x_2)  :=\underbrace{\omg_{L}(x_1 - d,x_2)\mathbf{1}_{\{x_1>0\}}}_{(\omg_{L,d})^{{r}}}
	-\underbrace{\omg_{L}(-x_1 - d, x_2)\mathbf{1}_{\{x_1<0\}}}_{(\omg_{L,d})^{{l}}}$$ for some $d>0$. In the theorem below, we 
	prove stability     when 
	a  pair of two opposite \textit{general} dipoles is initially close to the pair $\omg_{L,d}$ of  two opposite \textit{Lamb} dipoles 
	for large distance $d\gg1$ under $L^1$-bound. 
	\begin{theorem}\label{thm_lamb_nu} 
		For each  
  $\nu\geq 2\|\omg_L\|_{L^1}$
  and for  each $\varep_0>0$, 
		there exist $\dlt=\dlt(\nu,\varep_0)>0$ and $d_{0}=d_{0}(\nu,\varep_0)\geq 1$ such that for all $d\ge d_{0}$, if  
		we consider any odd-odd symmetric 
		initial data $\omg_0\in \mathcal{X} 
		$ with $\omg_0\geq0$ in $Q$ satisfying
		\begin{equation}\label{cond_lamb}
			\nrm{\omg_0}_{L^1}\leq \nu \quad\mbox{and}\quad\nrm{\omg_0-\omg_{L,d}}_{\mathcal{X}}\leq \dlt,
		\end{equation} 
		then 
		there exists a   function $\tau(\cdot_t):[0,\infty)\to[1,\infty)$ such that 
		the corresponding Euler solution $\omg(\cdot_t)$ 
		satisfies, for any $t\geq0$, \begin{equation*}
			\begin{split}
				\nrm{\omg(t) - \omg_{L,\tau(t)}}_{\mathcal{X}} \leq \varep_0. 
			\end{split}
		\end{equation*}  
	\end{theorem}
	\begin{remark}\label{rem_general_dip}
	 
    There are several other dipoles whose existence is obtained by a variational argument  (e.g. dipoles in  \cite{  BNL13, Burton.2021,   CQZZ-stab} and references therein). 
    In general,   their  stability is shown not for a single dipole (up to $x_1$-translation), but for the set of maximizers (i.e. the orbital sense), since it mostly remains open to verify uniqueness (up to $x_1$-translation) of a maximizer for a given variational setting except for very few cases (e.g.   the Lamb dipole \cite{AC2019, Wang.2024}). For the case without uniqueness, one may prove a similar stability result such as our Theorem  \ref{thm_lamb_nu} (and Theorem \ref{thm:main0}) for the pairs of opposite dipoles that are energy maximizers.
	\end{remark}
	\begin{proof} 
		
		Let $\varep_0>0$ and $\nu\geq 2\|\omg_L\|_{L^1}$.   
		We first take $\beta_0\in(0,1)$ such that
		for any $\tau\in[1-\beta_0,1]$,
		\begin{equation}\label{defn_alp0}
			\|\omega_{L,1}-\omg_{L,\tau}\|_{\mathcal{X}}\leq \frac{\varep_0}{2}.
		\end{equation}
		Denote \begin{equation}\label{defn_h}
			h(s):=\|\omg_L(\cdot-se_1) 
   \|_{L^2(Q)},\quad s\in\mathbb{R}.
		\end{equation}
		Then $h$ is continuous,
		$h(-1)=0,\, h(1)=\sqrt{\kappa_L}$, and it is strictly increasing on $[-1,1]$. Now we take
		sufficiently small $\varep\in(0,\varep_0)$ satisfying
		\begin{equation}\label{h_small}
			\sqrt{\kappa_L}-\varep\geq h(1-\beta_0).    
		\end{equation}
 
  	Then, we take $\theta=\theta(\varep,\nu):=F_\nu(\varep/4)>0$ 
   where $F_\nu$ is the function given in 
   Proposition \ref{prop:energy-estimate}. By monotonicity of the function $F_\nu$, it  guarantees		for any $s\geq0$,
		\begin{equation}\label{cont_inv2}
			F_\nu(s)<\theta\quad\implies \quad s < \frac{\varep}{4}.
		\end{equation} 
		Next we take two
		sufficiently small constants $$\alpha, \gamma\in\left(0,\frac{1}{8}\right)$$ by \eqref{conv_L1} 
		and by \eqref{est_en_diff} of Lemma \ref{lem_est_en_diff}
		such that
		whenever $\mu\in [(1-\alpha)\pi, (1+\alpha)\pi]$ and 
		$\kappa\in [(1-\gamma)\kappa_L, (1+\gamma)\kappa_L]$,  we get 
		\begin{equation}\label{est_L1}
			\nrm{\omg_{L}^{a,b} -\omg_{L}}_{L^1(\mathbb{R}^2)}\leq \frac{\varep}{7} \quad\mbox{and}\quad
			|E[\omg_{L}^{a,b}]-E[\omg_{L}]|
			\leq \frac{\theta}{24}
		\end{equation}
		for any $a=a(\mu,\kappa), b=b(\mu,\kappa)$   defined in  \eqref{formula_ab}. 
		We may assume $\alpha>0$ is small enough to satisfy
		\begin{equation}\label{est_alp_theta}
			\frac{3\alpha}{8}E[\omg_L]\leq \frac{\theta}{24}.
		\end{equation}
		Next, we choose 
		$d_0\geq 2$ large enough so that
		for any $d\geq d_0$,
		\begin{equation}\label{cond_large_d}
			E_{\text{inter}}[\omg_{L,d}]\leq \min\left\{\frac{\alpha}{4}  E[\omega_L],\, 
				\frac{\theta}{24}
				 \right\}
		\end{equation}
		thanks to the estimate \eqref{est_int_en}.\\
		
		We recall the notation  $\omg^{{r}}:=\omg\mathbf{1}_{\{x_1>0\}}$
		for any odd-odd
		symmetric $\omg$
		defined in \eqref{defn_w1}. By assuming $\dlt<1$, 		we also estimate, by the estimate \eqref{est_en_diff} of Lemma \ref{lem_est_en_diff} and by the assumption \eqref{cond_lamb},
		\begin{equation}\label{est_en_diff2}\begin{split}
				&|E[\omg_L]-E[\omg_0^{{r}}]|=
				|E[\omg_{L,d}\mathbf{1}_{x_1>0}]-E[\omg_0\mathbf{1}_{x_1>0}]|\\
				&\leq C\nrm{\omg_{L,d}-\omg_0}_{L^1(\{x_1>0\})}^{1/2}\cdot
				\nrm{\omg_{L,d}-\omg_0}_{L^2(\{x_1>0\})}^{1/2}\cdot
				\nrm{x_2(\omg_{L,d}+\omg_0)}_{L^1(\{x_1>0\})}^{1/2}\cdot
				\nrm{\omg_{L,d}+\omg_0}_{L^1(\{x_1>0\})}^{1/2}\\
				&\leq C(\nrm{\omg_{L,d}}_{L^1(\mathbb{R}^2)}+\nrm{\omg_0}_{L^1(\mathbb{R}^2)})\cdot
				\nrm{\omg_{L,d}-\omg_0}_{L^2(\mathbb{R}^2)}^{1/2}\cdot
				(\nrm{x_2(\omg_0-\omg_{L,d})}_{L^1(\mathbb{R}^2)}
				+2\nrm{x_2 \omg_{L,d}}_{L^1(\mathbb{R}^2)})^{1/2}\\
				&\leq C(1+\nu) \cdot
				\nrm{\omg_{L,d}-\omg_0}_{{\mathcal{X}}}^{1/2}\cdot
				(\nrm{\omg_{L,d}-\omg_0}_{{\mathcal{X}}}+1)^{1/2} 
				\leq C'(1+\nu) \cdot
				\dlt^{1/2},
		\end{split}\end{equation} where $C'>0$ is an absolute constant. Similarly,  we estimate
		\begin{equation*}\begin{split}
				&|E[\omg_{L,d}]-E[\omg_{0}]| 
				\leq 4C'(1+\nu) \cdot
				\dlt^{1/2}.
		\end{split}\end{equation*} 
		By observing
		\begin{equation*}
			\begin{split}
				2(E [\omg_{L}] - E_{\text{inter}}[\omg_{L,d}]) = E[\omg_{L,d}] 
				\quad\mbox{and}\quad  		2(E [\omg_0^{{r}}] - E_{\text{inter}}[\omg_{0}]) = E[\omg_0] ,
			\end{split}
		\end{equation*} 
		we have 	
		\begin{equation}\label{est_en_diff3}
			\begin{split}
				|  E_{\text{inter}}[\omg_{L,d}] - E_{\text{inter}}[\omg_{0}] |\leq 3C'(1+\nu) \cdot
				\dlt^{1/2}.
			\end{split}
		\end{equation} 
		
		We denote the constants
		$$
		\mu_0:=\int_Q x_2\omega_0\,dx  \quad \mbox{and}\quad\kappa_0:= \int_Q\omega_0^2\,dx,
		$$ 
		where the latter quantity is  preserved in time:
		$$ 
		\int_Q\omega(t)^2\,dx=\int_Q\omega_0^2\,dx.
		$$
		We also note 
		\begin{equation*}
			\begin{split}
				\mu(t):=
				\int_{Q} x_2 \omg(t) dx 
				=\frac{1}{2}\int_{\bbR^2  } x_2 \omg(t)^{{r}} dx  
				=\frac{1}{4}\int_{\bbR^2  } |x_2 \omg(t)| dx  
			\end{split}
		\end{equation*} 
		is   decreasing in time
		by Lemma \ref{lemma_center}.	
		By recalling
		$$
		\int_Q x_2\omega_{L,d}\,dx=\pi  \quad \mbox{and}\quad  \int_Q\omega_{L,d}^2\,dx=\kappa_L>0,$$ the condition \eqref{cond_lamb} implies \begin{equation}\label{est_L2_diff}\begin{split}
				|\mu_0-\pi|\leq \dlt\quad\mbox{and}\quad & |\sqrt{\kappa_0}-\sqrt{\kappa_L}|\leq \dlt.
		\end{split}\end{equation}
		Now we take  sufficiently small $\dlt\in(0,\min\{1,\varep/2\})$ satisfying
		\begin{equation}\label{cond_dlt}\begin{split}
				&    C'(1+\nu) \cdot
				\dlt^{1/2}\leq \frac{\alpha}{24}\cdot E[\omg_L],\quad\quad 
				(1-\alpha)\pi\leq\mu_0\leq (1+\alpha)\pi,\\ &
				\quad\mbox{and}\quad 
				\max\left\{\frac{1}{\sqrt{1+\gamma}}, \,\frac{1-\alpha}{1-(5/8)\alpha} \right\}\leq \sqrt{\kappa_L/\kappa_0}\leq \frac{1}{\sqrt{1-\gamma}}.
		\end{split} \end{equation}
		These conditions guarantee, by \eqref{est_en_diff2} and \eqref{est_en_diff3},
		\begin{equation}\begin{split}\label{est_en_diff4}
				&| E[\omg_L]-E[\omg_0^{{r}}]| \leq\frac{\alpha}{4}\cdot E[\omg_L],\quad \quad 
				|  E_{\text{inter}}[\omg_{L,d}] - E_{\text{inter}}[\omg_{0}] |\leq \frac{\alpha}{8}\cdot E[\omg_L],\\
				&	\quad\mbox{and}\quad
				\kappa_0\in [(1-\gamma)\kappa_L, (1+\gamma)\kappa_L].
		\end{split}  \end{equation}

		Let   $d\geq d_0$ and $\omg(\cdot_t)$ be the (odd-odd symmetric) solution (in $\mathbb{R}^2$) for the initial data $\omg_0$ for $t\geq 0$.
		Then we recall\begin{equation*}
			\begin{split}
				2(E [\omg_0^{{r}}] - E_{\text{inter}}[\omg_{0}]) = E[\omg_0] 
			\end{split}
		\end{equation*} and \begin{equation*}
			\begin{split}
				E[\omg_0] = E[\omg(t)] = 2(E[\omg(t)^{{r}}] - E_{\text{inter}}[\omg(t)]) \le  2 E[\omg(t)^{{r}}].
			\end{split}
		\end{equation*} 
		Therefore we have
		$$
		E[\omg_0^{{r}}]-E_{\text{inter}} [\omg_{0}]  \leq E[\omg(t)^{{r}}] 
		$$	which implies, by \eqref{est_en_diff4},	
		\begin{equation}\label{est_en_evol}
			E[\omg_{L}]-E_{\text{inter}} [\omg_{L,d}]  \leq E[\omg(t)^{{r}}] +\frac{3\alpha}{8}\cdot E[\omg_L].
		\end{equation}		
		Thus we get, by \eqref{cond_large_d},
		\begin{equation}\label{est_en_diff5}
			\begin{split}
				(1-\frac{5\alpha}{8})\cdot E[\omg_{L}]   \leq E[\omg(t)^{{r}}].
			\end{split}
		\end{equation} 
		
		On the other hand, 
		since we know 
		$$
		\omega(t)^{{r}}\in \mathcal{A}_{\mu(t),\kappa_0},
		$$ we have the estimate:	
		$$
		E[\omg(t)^{{r}}] \leq E[\omg_L^{a_t,b_t}]=\frac{\mu(t)}{\pi}\cdot\sqrt{\kappa_0/\kappa_L} \cdot E[\omg_L],
		$$ where $a_t:=a(\mu(t),\kappa_0), b_t:=b(\mu(t),\kappa_0)$ (see the definition of 
		$a(\cdot,\cdot), \,b(\cdot,\cdot)$	
		in \eqref{formula_ab}) and where the last identity is due to  \eqref{iden_en_scaling}.	
		Thus, by combining the above estimate with \eqref{est_en_diff5}, we get
		$$
		(1-\frac{5\alpha}{8})\cdot E[\omg_{L}] \leq \frac{\mu(t)}{\pi}\cdot\sqrt{\kappa_0/\kappa_L} \cdot E[\omg_L],
		$$ which implies
		$$
		\pi\sqrt{\kappa_L/\kappa_0} \cdot (1-\frac{5\alpha}{8}) \leq  \mu(t) \leq \mu_0  \quad \mbox{for any}\quad  t>0,
		$$ where  the  last inequality is due to the   monotonicity of 
		$\mu(t)$  by Lemma \ref{lemma_center}. The condition \eqref{cond_dlt} gives
		$$ 
		(1-\alpha)\pi \leq  \mu(t) \leq (1+\alpha)\pi \quad \mbox{for any}\quad  t>0,
		$$ 
		Now we estimate, by  \eqref{est_L1}, \eqref{est_en_evol}, \eqref{cond_large_d}, \eqref{est_alp_theta} and \eqref{est_en_diff4},
		$$ 
		E[\omg_{L}^{a_t,b_t}]- E[\omg(t)^{{r}}] \leq E[\omg_{L}]+\frac{\theta}{24}- E[\omg(t)^{{r}}] 
		\leq E_{\text{inter}}[\omg_{L,d}]	+\frac{3\alpha}{8}\cdot E[\omg_L]+\frac{\theta}{24}
		\leq \frac{\theta}{8} 
		\leq\frac{\mu(t)}{\pi}\cdot
		\sqrt{(\kappa_0/\kappa_L)}		
		\cdot \theta. 
		$$ 
		On the other hand, we have
		$$
		\nrm{\omg(t)^{{r}}}_{L^1}=\nrm{\omg_0^{{r}}}_{L^1}=\frac{1}{2}\nrm{\omg_0}_{L^1}\leq \frac{1}{2}\nu\leq \sqrt{(1-\alpha)}\cdot\left(\frac{\kappa_0}{\kappa_L}\right)^{1/4}\cdot \nu
		\leq \sqrt{\frac{\mu(t)}{\pi}}\cdot\left(\frac{\kappa_0}{\kappa_L}\right)^{1/4}\cdot \nu.
		$$
		Now we are ready to use the energy estimate \eqref{est_en_gen} to get
		$$
		F_{\nu}\left(
		\left( \max{\{(\mu(t)/\pi),\sqrt{(\kappa_0/\kappa_L)}\}	}\right)^{-1}\cdot
		\inf_{\tau\in\mathbb{R}}\nrm{\omg_{L}^{a_t,b_t}(\cdot -\tau  e_1) - \omg(t)^{{r}}}_{\mathcal{X}} \right)\leq \theta.
		$$ Hence, by \eqref{cont_inv2}, we get 
		$$
		\inf_{\tau\in\mathbb{R}}\nrm{\omg_{L}^{a_t,b_t}(\cdot -\tau  e_1)- \omg(t)^{{r}}}_{\mathcal{X}}  \leq   \max{\{(\mu(t)/\pi),\sqrt{(\kappa_0/\kappa_L)}\}	} \cdot
		\frac{\varep}{4} \leq    \max{\left\{(1+\alpha),\frac{1-(5/8)\alpha}{1-\alpha}\right\}	} \cdot
		\frac{\varep}{4} \leq  \frac{2\varep}{7}.
		$$
		By \eqref{est_L1}, we obtain
		$$ 
		\inf_{\tau\in\mathbb{R}}\nrm{\omg_{L}(\cdot -\tau  e_1) - \omg(t)^{{r}}}_{\mathcal{X}} \leq
		\left(\frac{1}{7}+\frac{2}{7}\right)  \varep
		\leq \frac{3\varep}{7}.
		$$
		Since $\omg(t)^{{r}}:=\omg(t)\mathbf{1}_{\{x_1>0\}}$, we know
		$$
		\nrm{\omg_{L}(\cdot -\tau(t)  e_1)\mathbf{1}_{\{x_1>0\}}  - \omg(t)^{{r}}}_{\mathcal{X}}   \leq \nrm{\omg_{L}(\cdot -\tau(t)  e_1)   - \omg(t)^{{r}}}_{\mathcal{X}}.
		$$ Then we take any $\tau(t)\in\mathbb{R}$ for each $t\geq 0$  satisfying
		\begin{equation}\label{tau_ep_2}
			\nrm{\omg_{L}(\cdot -\tau(t)  e_1)\mathbf{1}_{\{x_1>0\}} - \omg(t)^{{r}}}_{\mathcal{X}}   \leq \frac{ \varep}{2}\leq \frac{ \varep_0}{2}. 
		\end{equation}

		It only remains to verify that $\tau(t)$  satisfies $\tau(t)\geq 1$. 
		Indeed, by recalling \eqref{defn_h} and \eqref{est_L2_diff}, we observe 
		$$
  h(\tau(t))=\nrm{\omg_{L}(\cdot -\tau(t)  e_1) 
  }_{L^2(Q)} \geq \nrm{\omg(t)^{{r}}}_{L^2(Q)}  - \frac{ \varep}{2}
		= \sqrt{\kappa_0}  - \frac{ \varep}{2}\geq   \sqrt{\kappa_L} -\delta - \frac{ \varep}{2}\geq   \sqrt{\kappa_L} -\varep.
		$$  
		Thus, by \eqref{h_small}, we get
		$$\tau(t)\geq 1-\beta_0.$$
		Lastly, if $1-\beta_0\leq\tau(t)<1$, then we observe, by \eqref{tau_ep_2},
		\begin{equation*}
			\nrm{\omg_{L}(\cdot -  e_1)  - \omg(t)^{{r}}}_{\mathcal{X}}   \leq 
			\nrm{\omg_{L}(\cdot -  e_1)  - \omg_{L}(\cdot -  \tau(t) e_1)\mathbf{1}_{\{x_1>0\}}}_{\mathcal{X}}   
			+\frac{ \varep_0}{2}\leq \frac{\varep_0}{2}+\frac{\varep_0}{2}\leq \varep_0,
		\end{equation*} where the second inequality is due to \eqref{defn_alp0}.
		Thus, whenever $1-\beta_0\leq\tau(t)<1$, we simply redefine $\tau(t)=1$ which finishes the proof. \end{proof}

	\subsection{Proof of Theorem \ref{thm:main0}}
	
	Now we are  ready to prove Theorem \ref{thm:main0} by using Theorem \ref{thm_lamb_nu}.
	
	\begin{proof}[Proof of Theorem \ref{thm:main0}]
		
		First we take the constant $\nu_L:= 
  \|\omg_L\|_{L^1} 
  $, fix
		$\nu:=2\nu_L+1$, 	 
		and let $\varep>0$. We simply fix
		$\varep_0:=\varep/16$ and  take
		the constants $$\delta=\delta(\nu,\varep_0)>0\quad\mbox{and}\quad
		d_0=d_0(\nu,\varep_0)\geq 1$$ from Theorem \ref{thm_lamb_nu}. We may assume $\delta<1$ and $\delta<3\varep/32$. \\
		
		Now we consider  
		any odd-odd symmetric initial data $\omg_0\in \mathcal{X}$ with $\omg_0\geq0$ in $Q$ satisfying
		\eqref{cond_lamb_prime}.
		Since
		$$\|\omega_0\|_{L^1}\leq \|\omega_L\|_{L^1}+\delta\leq \nu_L+1\leq\nu,$$
		the initial data satisfies all the conditions \eqref{cond_lamb} of Theorem \ref{thm_lamb_nu}. 
		Thus, by the theorem,
		there exists a   function $\tau(\cdot_t):[0,\infty)\to[1,\infty)$ such that 
		the corresponding   solution $\omg(\cdot_t)$  
		satisfies, for any $t\geq0$, \begin{equation*}
			\begin{split}
				\nrm{\omg(t) - \omg_{L,\tau(t)}}_{\mathcal{X}} \leq \varep_0=\frac{\varep}{16}. 
			\end{split}
		\end{equation*}  
		
		Let $t\geq 0$ be fixed.   It remains to show
		$$  \nrm{\omg(t) - \omg_{L,\tau(t)}}_{L^1} \leq \frac{15\varep}{16}$$ or equivalently,
		$$  \nrm{\omg(t)^{{r}} - \omg_{L}(\cdot-\tau(t) e_{x_1})}_{L^1(\{x_1>0\})} \leq \frac{15\varep}{32}.$$
		We recall
		$\supp[\omega_L(\cdot-\tau(t) e_{x_1})=\overline{D^{\tau(t)}} $ where
		$$ {D^{\tau(t)}}:=\{x\in\mathbb{R}^2:\, |x-\tau(t) e_{x_1}|<1\}\subset \{x_1>0\}.$$ Thus we have
		\begin{align*}
			& \nrm{\omg(t)^{{r}} - \omg_{L}(\cdot-{\tau(t)} e_{x_1})}_{L^1(\{x_1>0\})} 
			\\&\quad\quad\leq \|\omega(t)^{{r}} -\omega_L(\cdot-{\tau(t)} e_{x_1})\mathbf{1}_{\{x_1>0\}}\|_{L^1(D^{\tau(t)})}+\|\omega(t)^{{r}} \|_{L^1(\mathbb{R}^2\setminus D^{\tau(t)})}  =:(I)+(II).
		\end{align*}
		For (I), we know
		$$
		(I)=\|\omega(t)-\omega_L(\cdot-{\tau(t)} e_{x_1})\|_{L^1(D^{\tau(t)}\cap \{x_1>0\})}
		\leq  \sqrt{|D^{\tau(t)}|}\cdot \|\omega(t)-\omega_L(\cdot-{{\tau(t)} e_{x_1}})\|_{L^2(D^{\tau(t)})}
		\leq  \sqrt{\pi}\cdot \frac{\varep}{16}.
		$$ 
		For (II), we know 
		\begin{align*}
			(II)=& \|\omega(t)^{{r}} \|_{L^1(\mathbb{R}^2\setminus D^{\tau(t)})}= \|\omega(t)^{{r}} \|_{L^1(\mathbb{R}^2 )}
			- \|\omega(t)^{{r}} \|_{L^1( D^{\tau(t)})}= \|\omega_0^{{r}} \|_{L^1(\mathbb{R}^2 )}
			- \|\omega(t)^{{r}} \|_{L^1( D^{\tau(t)})}
		\end{align*}  by conservation of $L^1$-norm.
		Then, we compute
		\begin{align*}
			(II)
			&\leq  \|\omega_0-\omega_{L,d} \|_{L^1(\{x_1>0\} )}+\|\omega_{L,d}\|_{L^1(\{x_1>0\} )}
			- \|\omega(t)^{{r}} \|_{L^1( D^{\tau(t)})}\\ 
			& \leq  \delta+\|\omega_L(\cdot-{{\tau(t)} e_{x_1}})\|_{L^1(D^{\tau(t)}\cap\{x_1>0\} )}
			- \|\omega(t) \|_{L^1( D^{\tau(t)}\cap\{x_1>0\})}\\ 
			&\leq  \delta  +\|\omega(t)-\omega_L(\cdot-{{\tau(t)} e_{x_1}})\|_{L^1(D^{\tau(t)}\cap\{x_1>0\})} \leq   \delta  +  \sqrt{\pi}\cdot\frac{\varep}{16}.
		\end{align*} 
		Lastly, due to $\delta<3\varepsilon/32$,  we are done.
	\end{proof}
	
	\begin{remark}
		It is natural to expect the shift position $\tau(t)$ has a similar speed  of one solo Lamb dipole $\omg_L$. For instance, for each $M>0$, one may prove,
		for $\|\omg\|_{L^\infty}\leq M$,
		$$\tau(t)\sim_M(d+t)$$ for sufficiently small $\varep>0$ 
		by following the same spirit of  \cite{CJ-Lamb} which proved dynamic stability of a ``solo'' Lamb dipole. 
	\end{remark}

	\section{Orbital stability for concentrated vortices}
	\label{sec:radial}
	In this section, we aim to prove Theorem~\ref{thm:main1}. Throughout this section, we assume the initial condition $\omega_0$ of \eqref{eq:Euler} is odd-odd in $\mathbb{R}^2$, and $\rho_0:=\omega_0\mathbf{1}_Q$ is non-negative, bounded, and satisfies $\int_{\mathbb{R}^2}\rho_0 dx = 1$. It is easy to check that $\omega(\cdot,t)$ stays odd-odd for all times, and let us define $\rho(\cdot,t) := \omega(\cdot,t) \mathbf{1}_{Q}$.

	\subsection{Decomposition of the kinetic energy functional}
	Like in Section~\ref{sec:dipole}, the conservation of kinetic energy also play a crucial role here. 
	Due to the odd-odd symmetry of $\omega$ and the definition of $\rho(\cdot,t)=\omega(\cdot,t)\mathbf{1}_{Q}$, we can rewrite the kinetic energy in \eqref{def_E} as
	\[
 \begin{split}
	E[\omega(t)] &=-\frac{1}{2\pi} \iint_{\mathbb{R}^2\times\mathbb{R}^2} \omega(x,t) \omega(y,t) \log|x-y| dydx\\
    &= \frac{2}{\pi} \iint_{\mathbb{R}^2\times\mathbb{R}^2} \rho(x,t) \omega(y,t) (-\log|x-y|) dydx\\
    &= \frac{2}{\pi} \iint_{\mathbb{R}^2\times\mathbb{R}^2} \rho(x,t) \rho(y,t)  K(x,y) dydx \\
    &=: \frac{2}{\pi}\tld E[\rho(t)],
    \end{split}
	\]
	where
	\[
	K(x,y) := \underbrace{(-\log |x-y|)}_{=: K_1(x-y)} +  \underbrace{\log |x-\tld y|}_{=: K_2(x,y)} +  \underbrace{(-\log |x+y|)}_{=:K_3(x,y)} + \underbrace{ \log |x-\bar y|}_{=:K_4(x,y)}. 
	\]
	The 4 terms $K_1,\dots, K_4$ come from the contributions from the four quadrants respectively, where $\tld y := (-y_1,y_2)$ and $\bar y = (y_1,-y_2)$. Such decomposition of $K$ allows us to decompose $\tld E[\rho(t)]$ into the sum
	\[
	\tld E[\rho(t)] =: E_1[\rho(t)] + E_2[\rho(t)]+ E_3[\rho(t)] + E_4[\rho(t)].
	\]
 
	\medskip
	In this subsection, we will use two different viewpoints to decompose $\tld E[\rho(t)]$ into two parts, and obtain various estimates. In the proof, we will go back and forth between these two viewpoints.

	\medskip
	\noindent\textbf{Viewpoint 1: }Decomposing $\tld E[\rho(t)]$ into ``self-interaction'' and ``others''.
	
	\medskip
	\noindent  \textbf{$\circ$ Contribution from self-interaction:}
	\smallskip
	
	Note that $E_1[\rho(t)]$ is contributed by the ``self-interaction'' of vorticity in the first quadrant, and the kernel $K_1$ only depends on $x-y$. 
	
	We introduce the following notation for such type of energy: for any potential $W \in L_{loc}^1(\mathbb{R}^2)$, we define the interaction energy of $\rho$ (with interaction potential $W$) as
	\begin{equation}\label{inter}
		\mathcal{E}_W[\rho] := \iint_{\mathbb{R}^2\times\mathbb{R}^2} \rho(x)\rho(y)W(x-y)dxdy.
	\end{equation}
	With this notation, $E_1[\rho(t)]$ can also be written as $\mathcal{E}_{K_1}[\rho(t)]$. In particular, since $K_1=-\log(|\cdot|)$ is a radially decreasing function, for $\rho_0 \in L^1_+(\mathbb{R}^2)$ (the set of non-negative $L^1$ functions), Riesz rearrangement inequality (see \cite[Sec 3.7]{LL} for example) gives
	\[
	\mathcal{E}_{K_1}[\rho(t)] \leq \mathcal{E}_{K_1}[\rho(t)^*] =  \mathcal{E}_{K_1}[\rho^*_0] \quad \text{ for all }t\in\mathbb{R},
	\]
	where the equality is due to $\rho(\cdot,t) \in L^1_+(\mathbb{R}^2)$ has the same distribution function as $\rho_0$ (thus $\rho(\cdot,t)^* = \rho_0^*$), since $\rho$ is transported by a divergence-free vector field.
	
	The Riesz rearrangement inequality above can be upgraded into a quantitative version (also called ``stability estimate'') $\mathcal{E}_{K_1}[\rho^*] - \mathcal{E}_{K_1}[\rho] \geq d(\rho,\rho^*)$, where $d(\rho,\rho^*)\geq 0$ measures the ``asymmetry'' of $\rho$ in some sense. Such stability estimates were obtained in \cite{FP, FL, BC} when $\rho$ is a characteristic function, and in \cite{YY} for a general non-negative density. Below we state the theorem in \cite{YY} (we only state the result in dimension 2 for our application), which will be used later in the proof.
	\begin{theorem}[Theorem 1.1 of \cite{YY}]
		\label{thm_yy}
		Consider a radial interaction potential $W \in C^1(\mathbb{R}^2\setminus\{0\})\cap L_{loc}^1(\mathbb{R}^2)$. Let $w:\mathbb{R}^+\to\mathbb{R}$ be such that $w(|x|)=W(x)$. Assume $w'(r)<0$ for $r>0$, and there exists some $c>0$ such that $w'(r)\leq -cr$ for $r\in(0,1)$.
		
		For all $\rho \in L^1_+(\mathbb{R}^2) \cap L^\infty(\mathbb{R}^2)$ that satisfies $\supp\rho^* \subset B(0,R_*)$ with $R_*$ being finite, we have the following stability estimate for the interaction energy \eqref{inter}:
		\[
		\mathcal{E}_W[\rho^*] - \mathcal{E}_W[\rho] \geq c(W,R_*) \|\rho\|_{L^1}^{3} \|\rho\|_{L^\infty}^{-1} \delta(\rho)^2
		\]
		for some strictly positive constant\footnote{More precisely, in the proof of \cite[Theorem 1.1]{YY}, the constant $c(W,R_*)$ is explicitly given by
			$
			c(W,R_*) := \inf_{r\in(0,20R_*)} -\frac{w'(r)}{r}.
			$ We will not use it in this paper.
		} $c(W,R_*)$, where 
		\[
		\delta(\rho) := \inf_{a\in\mathbb{R}^2} \frac{\|\rho-\rho^*(\cdot-a)\|_{L^1(\mathbb{R}^2)}}{2\|\rho\|_{L^1(\mathbb{R}^2)}}.
		\]
	\end{theorem}

	\noindent  \textbf{$\circ$ Contribution from others:}
	
	Let us denote 
	\begin{equation}
		\label{E234_def}
		E_{234}[\rho(t)] := \iint_{\mathbb{R}^2\times \mathbb{R}^2}\rho(x,t) \rho(y,t) K_{234}(x,y) dxdy,
	\end{equation}
	where \begin{equation}\label{K234_def}
		K_{234}(x,y) := K_2(x,y) + K_3(x,y) + K_4(x,y).
	\end{equation}
	Below we point out an important property of the potential $K_{234}$: for any $p\in Q$ and $q\in\mathcal{O}_{pv}(p)$, we have 
	\begin{equation}\label{K234_eq}
		K_{234}(p,p) = K_{234}(q, q).
	\end{equation}
	It follows from a direct computation using the formula for $\mathcal{O}_{pv}(p)$ given in \eqref{eq:orbit-pv}. Alternatively, $K_{234}(p,p)$ is the Hamiltonian for the point vortex motion of a single vortex in $Q$, which is conserved; see \cite[Section 2.3]{Yang} for details. 
	
	\medskip
	\noindent\textbf{Viewpoint 2: }Decomposing $\tld E[\rho(t)]$ into ``left dipole'' and ``right dipole'' contributions.
	
	\medskip
	\noindent \textbf{$\circ$ Contribution from the dipole on the left:} 
	
	For any $x,y \in Q$, one can easily check that $|x+y|\geq |x-\tld y|$, thus $K_2(x,y)+K_3(x,y)< 0$. Therefore, if $\rho_0\geq 0$ in $Q$, we have $\rho(\cdot,t)\geq 0$ in $Q$, leading to
	\begin{equation}
		\label{observ}
		E_2[\rho(t)]+E_3[\rho(t)]\leq 0 \quad\text{ for all }t\in\mathbb{R}.
	\end{equation}

	\noindent  \textbf{$\circ$ Contribution from the dipole on the right:}

	Next we move on to the sum $E_1 + E_4$, which comes from the contribution from the right half plane, and it is a double integral with potential $K_1+K_4 = \log\frac{|x-\bar y|}{|x-y|}$. Let us state and prove the following lemma regarding a pointwise upper bound for this potential. Although elementary, it will play a central role in the proof.
	
	\begin{lemma}[Key Lemma]\label{lemma_K} Let $x = (x_1,x_2)$ with $x_2 > 0$. 
		For any $H>300$, we have 
		\[
		\log \frac{|x-\bar y|}{|x-y|} \leq \varphi(|x-y|) + g^H(2x_2) + \frac{300}{H}, 
		\]
		where $\varphi \in C((0,+\infty))$ is defined as
		\[
		\varphi(s) := \begin{cases}
			-\log s & \text{ for }0 < s < 3,\\
			-\log 3& \text{ for }s\geq 3,
		\end{cases}
		\]
		and $g^H \in C((0,+\infty))$ is defined as
		\begin{equation}
			\label{def_gH}
			g^H(s) := \begin{cases}
				\log s & \text{ for }s \geq H,\\
				\log H + \frac{s-H}{H} & \text{ for }0<s<H.
			\end{cases}
		\end{equation}
	\end{lemma}
			\begin{figure}[htbp]
		\begin{center}
		\hspace*{-1cm}\includegraphics[scale=1]{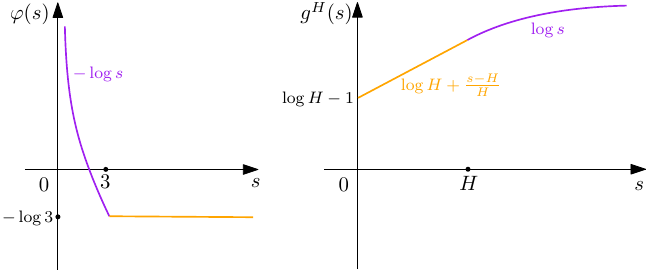}
				\caption{\label{fig_functions}
		Illustrations of the functions $\varphi$ and $g^H$ in Lemma~\ref{lemma_K}. Note that both functions are continuous in $(0,\infty)$. In addition, $g^H$ has its left and right derivative at $H$ matching each other, and is concave in $(0,\infty)$.}
		\end{center}
		\end{figure}

	\begin{proof}
		Let us denote $d_1 := x_1-y_1$ and $d_2 := x_2-y_2$. With such notations, we have
		\[
		\log \frac{|x-\bar y|}{|x-y|} = \frac{1}{2} \log \frac{d_1^2 + (2x_2-d_2)^2}{d_1^2 + d_2^2}.
		\]
		The proof is divided into the following two cases:
		
		\medskip
		
		\noindent\textbf{Case 1.} $|x-y|=\sqrt{d_1^2 + d_2^2} < 3$. In this case, note that $\varphi(|x-y|)=-\log|x-y|$, thus it suffices to prove that $\log |x-\bar y| \leq g^H(2x_2)+\frac{300}{H}$.
		To show this,  we discuss two sub-cases depending on $x_2$:
		
		\begin{itemize}
			\item If $x_2 \geq \frac{H}{100}$, we rewrite $\log|x-\bar y|$ as
			\[
			\begin{split}
				\log|x-\bar y| &= \frac{1}{2} \log(d_1^2 + (2x_2 - d_2)^2)\\
				&= \log(2x_2-d_2) + \frac{1}{2} \log\left(1+\frac{d_1^2}{(2x_2-d_2)^2}\right)\\
				&\leq  \log(2x_2+3) + \frac{1}{2} \log\left(1+\frac{3^2}{(2x_2-3)^2}\right)\\
				&= \log(2x_2) + \log\left(1+\frac{3}{2x_2}\right) + \frac{1}{2} \log\left(1+\frac{3^2}{(2x_2-3)^2}\right).
			\end{split}
			\]
			Since $H>300$ and $x_2\geq\frac{H}{100}>3$, we have $2x_2-3\geq x_2$, so
			\[
			\log|x-\bar y| \leq \log(2x_2) +\frac{3}{2x_2} + \frac{9}{2x_2^2} \leq \log(2x_2) +\frac{300}{H}.
			\]
			\item 
			If $x_2 < \frac{H}{100}$, using that $(2x_2-d_2)^2 \leq 8x_2^2 + 2d_2^2$, we directly bound $\log|x-\bar y|$ as
			\begin{equation}\label{temp0}
				\begin{split}
					\log|x-\bar y| 
					&\leq \frac{1}{2} \log (d_1^2 + 2d_2^2 + 8x_2^2)\\
					&\leq \frac{1}{2} \log\left(18 + \frac{8H^2}{10000}\right)\\
					&\leq \frac{1}{2} \log\left(\frac{H^2}{1000}\right)\\
					& < g^H(0) \\
					&\leq g^H(2x_2).\\
				\end{split}
			\end{equation}
			Here the third inequality follows from the fact that $H>300$, and the last inequality follows from the fact that $g^H(s)$ is an increasing function in $s$.
		\end{itemize}
		
		\textbf{Case 2.} $|x-y|=\sqrt{d_1^2 + d_2^2} \geq 3$. In this case, have
		\begin{equation}\label{temp1}
			\begin{split}
				\log\frac{|x-\bar y|}{|x-y|} &\leq \frac{1}{2} \log \frac{d_1^2+2d_2^2 + 4x_2 d_2 + 4 x_2^2}{d_1^2+d_2^2} \leq \frac{1}{2}\log\left(2 + \frac{4x_2d_2 + 4x_2^2}{d_1^2+d_2^2}\right).\\ 
			\end{split}
		\end{equation}
		Below we discuss two sub-cases:
		\begin{itemize}
			\item If $\frac{x_2}{\sqrt{d_1^2+d_2^2}} < \frac{H}{100}$, then \eqref{temp1} becomes
			\[
			\log\frac{|x-\bar y|}{|x-y|} \leq \frac{1}{2}\log\left(2 + \frac{4H}{100} + \frac{4H^2}{10000}\right)\leq \frac{1}{2} \log\left(\frac{H^2}{1000}\right)\leq \varphi(|x-y|) + g^H(2x_2),
			\]
			where the last step follows from a similar argument as the last three inequalities in \eqref{temp0}, and we also use that $\varphi(|x-y|)\geq -\log 3$.
			
			\item If $\frac{x_2}{\sqrt{d_1^2+d_2^2}} \geq \frac{H}{100}$, then \eqref{temp1} becomes
			\[
			\begin{split}
				\log\frac{|x-\bar y|}{|x-y|} &\leq \frac{1}{2}\log\left(\Big(4+\frac{500}{H}\Big)\frac{x_2^2}{d_1^2+d_2^2}\right) \\
				&= \log(2x_2)-\log(|x-y|) +\frac{1}{2}\log\left(1+\frac{125}{H}\right)\\
				&\leq   \varphi(|x-y|) + g^H(2x_2) + \frac{300}{H},
			\end{split}
			\]
			where the last step follows from   $\varphi(|x-y|) \geq -\log(|x-y|)$ and $g^H(2x_2)\geq \log(2x_2)$.\qedhere 
		\end{itemize}
	\end{proof}

	\subsection{Forward-in-time orbital stability for far-away initial data}
	
	In this subsection, we prove a key result that lead to orbital stability for all positive times, under the assumptions that the initial vortex has compact support whose area is of order 1, and satisfies various inequalities of the energy functional $E_1,\dots,E_4$. These assumptions look rather technical, but roughly speaking, they correspond to the setting where the initial data $\rho_0$ is centered near a far-away point $(l_0,h_0)$ with $l_0\gg h_0 \gg 1$, and it is already very close to a translation of its radially decreasing rearrangement $\rho_0^*$. To help the readers visualize these assumptions, in Remark~\ref{rmk_symm} afterwards, we give a concrete example of $\rho_0$ that satisfies these assumptions.

	\begin{theorem}
		\label{thm:to-infty}
		Let $\omega_0$ be an odd-odd initial data to \eqref{eq:Euler}.
		For any $\varep \in (0, \frac{1}{1000})$ and $A\geq 1$, assume $\rho_0 = \omega_0 \mathbf{1}_Q$ satisfies the following four assumptions:
		\begin{equation}\label{a1}
			\int_{\mathbb{R}^2} \rho_0(x) dx = 1, \quad 0\leq\rho_0 \leq A \quad  \text{ and }\quad |\supp \rho_0| \leq \pi;
		\end{equation}
		\begin{equation}\label{a2}E_1[\rho_0] - E_1[\rho_0^*] \geq - \varep;
		\end{equation}
		\begin{equation}\label{a3}
			X_{02} := \int_{\mathbb{R}^2}\rho_0(x) x_2 dx > \varep^{-1}, \quad\text{and}\quad E_4[\rho_0] - \log(2X_{02}) \geq - \varep;
		\end{equation}
		\begin{equation}\label{a4}
			E_2[\rho_0] + E_3[\rho_0] \geq- \varep.
		\end{equation}
		
		Then for all $t\geq 0$, $\rho(\cdot,t) := \omega(\cdot,t) \mathbf{1}_Q$ satisfies
		\begin{equation}\label{conclusion1}
			\inf_{a\in\mathbb{R}^2} \|\rho(\cdot,t) - \rho_0^*(\cdot-a)\|_{L^1(\mathbb{R}^2)} \leq C_0 \sqrt{A\varep},
		\end{equation}
		where $C_0$ is a universal constant.
		In addition, its vertical center of mass $X_2(t) := \int_{\mathbb{R}^2} \rho(x,t) x_2 dx$ satisfies
		\begin{equation}\label{conclusion2}
			\frac{X_2(t)}{X_{02}} \in (1-400 \varep, 1] \quad\text{ for all } t\geq 0.
		\end{equation}
		
	\end{theorem}
	
	\begin{proof}
		From the conservation of kinetic energy, we have $\tld E[\rho_0] = \tld E[\rho(t)]$ for all $t\in\mathbb{R}$. In what follows, we shall obtain a lower bound of $\tld E[\rho_0]$, and an upper bound of $\tld E[\rho(t)]$ for $t\geq 0$. 
		
		\medskip
		\noindent\textbf{Lower bound of energy.}	To obtain a lower bound of $\tld E[\rho_0]$, we directly apply the assumptions \eqref{a2}--\eqref{a4} to obtain
		\begin{equation}
			\label{lowerbd}
			\tld E[\rho_0] \geq E_1[\rho_0^*] + \log(2X_{02}) - 3\varep.
		\end{equation}
		
		\medskip
		\noindent\textbf{Upper bound of energy.}	 Next we will bound $\tld E[\rho(t)]$ from above for all $t\geq 0$. Using the observation \eqref{observ}, we have $E_2[\rho(t)] + E_3[\rho(t)]\leq 0$, therefore
		\[
		\tld E[\rho(t)] \leq E_1[\rho(t)] + E_4[\rho(t)] =  \iint_{\mathbb{R}^2\times\mathbb{R}^2}\rho(x,t)\rho(y,t) \log \frac{|x-\bar y|}{|x-y|}  dxdy.
		\]
		To control the right hand side, we apply Lemma~\ref{lemma_K} (with $H$ replaced by $X_{02}$; recall that \eqref{a3} gives $X_{02}\geq\varep^{-1}>1000$) to obtain a pointwise upper bound on the integrand. This leads to
		\begin{equation}
			\label{eq2}
			\begin{split}
				\tld E[\rho(t)] 
				& \leq \iint_{\mathbb{R}^2\times \mathbb{R}^2}\rho(x,t) \rho(y,t)  \varphi(|x-y|) dxdy+ \int_{\mathbb{R}^2} \rho(x,t) g^{X_{02}}(2x_2)dx + \frac{300}{X_{02}}.
			\end{split}
		\end{equation}
		We will obtain an upper bound for each of the three terms on the right hand side: For the last term, our assumption \eqref{a3} directly lead to
		\begin{equation}\label{term3}
			\frac{300}{X_{02}} \leq 300 \varep.
		\end{equation}
		
		For the second term, since $g^{X_{02}}$ is concave and $\int_{\mathbb{R}^2}\rho(x,t)dx=1$, applying Jensen's inequality, we have
		\begin{equation}\label{term2}
			\int_{\mathbb{R}^2} \rho(x,t) g^{X_{02}}(2x_2)dx  \leq g^{X_{02}}(2X_2(t)).
		\end{equation} 
		
		Finally, for the first term, let us first define $\tilde \varphi \in C^1((0,+\infty))$ such that $\tilde\varphi(s) = \varphi(s)$ for $0<s\leq 2$,  $\tilde\varphi(s) \geq \varphi(s)$ for $s\in[2,\infty)$, and $\tilde\varphi'(s)<0$ for all $s>0$.  Let us also define $\tilde\Phi:\mathbb{R}^2\setminus\{0\}\to\mathbb{R}$ such that $\tilde\Phi(x) := \tilde\varphi(|x|)$. Since $\tilde\varphi \geq \varphi$ and $\rho(\cdot,t)\geq 0$, clearly we have
		\begin{equation}\label{term1}
			\iint_{\mathbb{R}^2\times \mathbb{R}^2}\rho(x,t) \rho(y,t)  \varphi(|x-y|) dxdy\leq \iint_{\mathbb{R}^2\times \mathbb{R}^2}\rho(x,t) \rho(y,t)  \tilde\varphi(|x-y|) dxdy = \mathcal{E}_{\tilde\Phi}[\rho(t)].
		\end{equation}

		Such $\tilde\Phi$ allows us to apply Theorem~\ref{thm_yy} to obtain   
		\begin{equation}\label{temp123}
			\mathcal{E}_{\tilde\Phi}[\rho^*(t)] - \mathcal{E}_{\tilde\Phi}[\rho(t)] \geq c(\tilde\Phi, R_*) \|\rho(\cdot,t)\|_{L^\infty}^{-1} \inf_{a\in\mathbb{R}^2} \|\rho(\cdot,t)-\rho^*(\cdot-a,t)  \|_{L^1(\mathbb{R}^2)}^2,
		\end{equation}
		where $R_*$ is the radius of support of $\rho^*(\cdot,t)$.
		Since $\rho(\cdot,t)$ has the same distribution as $\rho_0$, we have $\rho^*(\cdot,t) = \rho_0^*$ for all $t$, and $\|\rho(\cdot,t)\|_{L^\infty} = \|\rho_0\|_{L^\infty} \leq A$ (the last step is by \eqref{a1}).  Also by \eqref{a1}, $\rho_0^*$ is supported in the unit disk (so $R_*=1$), therefore $c(\tilde\Phi, R_*)$ is a universal constant since $\tilde\Phi$ is fixed. Let us denote it by $c_0$. With these observations, we rewrite \eqref{temp123} as
		\begin{equation}
			\label{term12}
			\mathcal{E}_{\tilde\Phi}[\rho(t)] \leq \mathcal{E}_{\tilde\Phi}[\rho^*_0] - c_0 A^{-1} \inf_{a\in\mathbb{R}^2} \|\rho(\cdot,t)-\rho_0^*(\cdot-a) \|_{L^1(\mathbb{R}^2)}^2.
		\end{equation}
		
		Next we make another elementary but useful observation. Even though the interaction potential  $\tilde \Phi(x-y)$ is different from $K_1(x-y)$, they do agree with each other when $|x-y|\leq 2$ (recall that $\tilde\varphi(s) = \varphi(s) = -\log s$ when $0<s\leq 2$). Since $\supp\rho_0^*\subset \overline{B(0,1)}$ (which follows from the assumption $|\supp\rho_0|\leq \pi$ in \eqref{a1}), any two points $x,y \in \supp\rho_0^*$ satisfy $|x-y|\leq 2$, therefore
		\begin{equation}\label{term13}
			E_1[\rho_0^*] = \mathcal{E}_{K_1}[\rho_0^*] = \mathcal{E}_{\tilde\Phi}[\rho_0^*].
		\end{equation}
		
		Finally, we apply \eqref{term3}, \eqref{term2}, \eqref{term1}, \eqref{term12} and \eqref{term13} to \eqref{eq2}, and arrive at
		\begin{equation}\label{eq3}
			\tld E[\rho(t)] \leq E_1[\rho_0^*] - c_0 A^{-1} \inf_{a\in\mathbb{R}^2} \|\rho(\cdot,t)-\rho_0^*(\cdot-a) \|_{L^1(\mathbb{R}^2)}^2  + g^{X_{02}}(2X_2(t)) + 300\varep.
		\end{equation}

		\medskip
		\noindent \textbf{Completion of the proof}. Finally, let us put together the lower bound \eqref{lowerbd} and the upper bound \eqref{eq3}, and use the fact that $\tld E[\rho_0] = \tld E[\rho(t)]$. They lead to
		\begin{equation}\label{eq_final}
			c_0 A^{-1}	\inf_{a\in\mathbb{R}^2} \|\rho(\cdot,t)-\rho_0^*(\cdot-a) \|_{L^1(\mathbb{R}^2)}^2 + \log(2X_{02}) - g^{X_{02}}(2X_2(t)) \leq 303 \varep.
		\end{equation}
		By Lemma~\ref{lemma_center}, we have $X_2(t)\leq X_{02}$ for all $t\geq 0$, therefore 
		\[
		g^{X_{02}}(2X_2(t)) \leq g^{X_{02}}(2X_{02}) = \log(2X_{02}) \quad\text{ for all }t\geq 0,
		\]
		where we used the definition of $g^H$ in \eqref{def_gH} with $H=X_{02}$.
		
		Plugging this into \eqref{eq_final} gives
		\[
		c_0 A^{-1}	\inf_{a\in\mathbb{R}^2} \|\rho(\cdot,t)-\rho_0^*(\cdot-a) \|_{L^1(\mathbb{R}^2)}^2\leq 303\varep \quad\text{ for all }t\geq 0,
		\]
		leading to \eqref{conclusion1}.
		
		To show \eqref{conclusion2}, note that \eqref{eq_final} also implies
		\begin{equation}\label{temp_log}
			\log(2X_{02}) - g^{X_{02}}(2X_2(t)) \leq 303 \varep \quad\text{ for all }t\geq 0.
		\end{equation}
		By the concavity of $g^{X_{02}}$ and the facts that $g^{X_{02}}(2X_{02})=\log(2X_{02})$ and $(g^{X_{02}})'(2X_{02})=\frac{1}{2X_{02}}$, we have 
		\[
		g^{X_{02}}(s)  \leq \log(2X_{02}) - \frac{2X_{02}-s}{2X_{02}} \quad\text{ for all } s \in [0, 2X_{02}].
		\]
		In particular, since we have $X_2(t)\leq X_{02}$ for all $t\geq 0$ by Lemma~\ref{lemma_center}, we can set $s=2X_2(t)$ in the above inequality, which becomes
		\[
		g^{X_{02}}(2X_2(t))  \leq \log(2X_{02}) - \frac{X_{02}-X_2(t)}{X_{02}} \quad\text{ for all }t\geq 0.
		\]
		Plugging this into \eqref{temp_log} yields
		\[
		\frac{X_2(t)}{X_{02}} \geq 1 - 303\varep \quad \text{ for all }t\geq 0,
		\]
		finishing the proof of \eqref{conclusion2}.
	\end{proof}

\begin{figure}[htbp]
		\begin{center}
		\hspace*{-1cm}\includegraphics[scale=1]{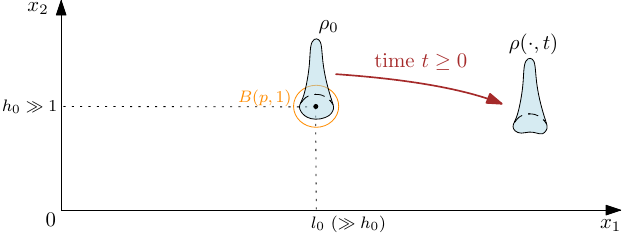}
				\caption{\label{fig_orbit2}
		Illustration of the initial data $\rho_0$ (and solution at a later time) in Remark~\ref{rmk_symm}.}
		\end{center}
		\end{figure}
	
	\begin{remark}
		\label{rmk_symm}
		To help the readers visualize the assumptions \eqref{a1}--\eqref{a4}, we give a concrete example of $\rho_0$ that satisfies them. Let us take any $\rho_0 = \rho_0^*(\cdot-p)$ that satisfies \eqref{a1}. Since such $\rho_0$ is a translation of $\rho_0^*$, \eqref{a2} is automatically satisfied. Furthermore, assume $p=(l_0, h_0)$ is such that $h_0 \geq 2\varep^{-1}$ and $l_0 \geq 2\varep^{-1} h_0$. See Figure~\ref{fig_orbit2} for an illustration.
  
  Since such $\rho_0$ is supported inside $B(p,1)$, let us check that \eqref{a3}--\eqref{a4} are also satisfied: for \eqref{a3}, note that $X_{02} = h_0 > \varep^{-1}$, and for any $x,y\in\supp\rho_0$, we have $|x-\bar y|>2h_0-2$. This leads to 
  \[E_4[\rho_0] -\log(2X_{02})\geq \log(2h_0-4)-\log(2h_0) \geq \log\left(1-\frac{1}{2h_0}\right)\geq -\varep.
  \]
  To check \eqref{a4}, note that
  \[
  E_3[\rho_0]+ E_4[\rho_0] \geq \inf_{x,y\in\supp\rho_0} \log\frac{|x-\tld y|}{|x+y|}
\geq \log\frac{l_0-2}{l_0+h_0+2} \geq -\varep,
\]
where the last inequality follows from our assumptions $h_0\geq 2\varep^{-1}, l_0 \geq 2\varep^{-1}h_0$.

	\end{remark}
	
	\medskip
	
	We now derive a rescaled version of Theorem~\ref{thm:to-infty} for concentrated vortices supported in a small area. Note that the assumptions and conclusions are mostly identical to \eqref{thm:to-infty}, except for the appearance  of the scaling factor $\lambda$ in \eqref{a1new} and \eqref{a3new}.
	\begin{corollary}
		\label{cor:rescale}
		Let $\omega_0$ be an odd-odd initial data to \eqref{eq:Euler}. 
		For any $\varep \in (0, \frac{1}{1000}), A\geq 1$ and $\lambda\in (0,1)$, assume $\rho_0 = \omega_0 \mathbf{1}_Q$ satisfies the following four assumptions:
		\begin{equation}\label{a1new}
			\int_{\mathbb{R}^2} \rho_0(x) dx = 1, \quad 0 \leq \rho_0 \leq A \lambda^{-2} \quad  \text{ and }\quad |\supp \rho_0| \leq \pi \lambda^2;
		\end{equation}
		\begin{equation}\label{a2new} E_1[\rho_0]  - E_1[\rho_0^*] \geq - \varep;
		\end{equation}
		\begin{equation}\label{a3new}
			X_{02} := \int_{\mathbb{R}^2}\rho_0(x) x_2 dx > \varep^{-1}\lambda, \quad\text{and}\quad E_4[\rho_0] - \log(2X_{02}) \geq - \varep;
		\end{equation}
		\begin{equation}\label{a4new}
			E_2[\rho_0] + E_3[\rho_0] \geq- \varep.
		\end{equation}
		
		Then $\rho(\cdot,t) := \omega(\cdot,t) \mathbf{1}_Q$ satisfies
		\[
			\inf_{a\in\mathbb{R}^2} \|\rho(\cdot,t)-\rho_0^*(\cdot-a)\|_{L^1(\mathbb{R}^2)} \leq C_0 \sqrt{A\varep} \quad\text{ for all }t\geq 0,
		\]
		where $C_0$ is a universal constant.
		In addition, its vertical center of mass $X_2(t) := \int_{\mathbb{R}^2} \rho(x,t) x_2 dx$ satisfies
		\[
			\frac{X_2(t)}{X_{02}} \in (1-400 \varep, 1] \quad\text{ for all } t\geq 0.
		\]
		
	\end{corollary}
 
	\begin{proof}
		We recall the following scale invariance of the 2D Euler equations: if $\omg(x,t)$ is a solution of \eqref{eq:Euler}, then for any $\lambda>0$, \begin{equation*}
			\begin{split}
				\tld{\omg}(x,t) := \lambda^{2} \omg( \lambda x, \lambda^{2} t )
			\end{split}
		\end{equation*} is again a solution. 
		Let us check that when $\rho_0 = \omega_0 \mathbf{1}_Q$ satisfies \eqref{a1new}--\eqref{a4new}, $\tld\rho_0 = \tld{\omg}_0 \mathbf{1}_Q$ satisfies all assumptions \eqref{a1}--\eqref{a4} of Theorem~\ref{thm:to-infty}:
		\begin{itemize}
			\item One can easily check that such scaling leaves the $L^1$ norm invariant, so 
			$
			\int_{\mathbb{R}^2} \tld\rho_0 dx = \int_{\mathbb{R}^2} \rho_0 dx.
			$
			In addition, we have $\|\tld\rho_0\|_{L^\infty} = \lambda^{2} \|\rho_0\|_{L^\infty}$, and $|\supp \tld\rho_0| = \lambda^{-2} |\supp\rho_0|$. Hence if $\rho_0$ satisfies \eqref{a1new}, we have $\tld\rho_0$ satisfies \eqref{a1}.
			\item For the vertical center of mass, the scaling and \eqref{a3new} gives 
			\[
			\tld{X}_{02} := \int_{\mathbb{R}^2} \tld\rho_0(x) x_2 dx = \lambda^{-1} \int_{\mathbb{R}^2} \rho_0(x) x_2 dx = \lambda^{-1}X_{02} > \varep^{-1},
			\] 
			so $\tld\rho_0$ satisfies the first assumption of \eqref{a3}.
			\item For the energy, the scaling gives
			\[
			\tld E[\tld\rho_0] = \iint_{\mathbb{R}^2\times\mathbb{R}^2} \tld\rho_0(x)\tld\rho_0(y) K(x,y) dxdy = \iint_{\mathbb{R}^2\times\mathbb{R}^2} \rho_0(x)\rho_0(y) K(\lambda^{-1} x,\lambda^{-1} y) dxdy,
			\]
			and likewise we can replace $K$ by $K_i$ to get the scaling for each $E_i[\tld\rho_0]$.
			Namely, since $K_1,\dots, K_4$ contains different signs of $\log$, we have
			\[
			E_i[\tld\rho_0] = \begin{cases}
				E_i[\rho_0] + \log\lambda &\text{ for } i= 1,3;\\
				E_i[\rho_0] - \log\lambda &\text{ for } i= 2,4.
			\end{cases}
			\]
			Similarly, we also have $E_1[(\tld\rho_0)^*] = E_1[\rho_0^*] + \log\lambda$.
			
			Combining these identities with the assumptions \eqref{a2new}--\eqref{a4new} yields
			\[
			\begin{split}		&E_1[\tld\rho_0]-E_1[(\tld\rho_0)^*] = (E_1[\rho_0] + \log \lambda) -  (E_1[\rho_0^*] + \log \lambda)\geq - \varep;
				\\
				&E_4[\tld\rho_0] - \log(2\tld{X}_{02}) = (E_4[\rho_0]-\log\lambda) - \log(2\lambda^{-1}X_{02}) = E_4[\rho_0]-\log(2X_{02})\geq - \varep;
				\\[0.1cm]
				&E_2[\tld\rho_0] + E_3[\tld\rho_0] = (E_2[\rho_0] - \log\lambda) + (E_3[\rho_0] + \log\lambda) =E_2[\rho_0]  + E_3[\rho_0] \geq - \varep,
			\end{split}
			\]
			so  $\tld\rho_0$ satisfies \eqref{a2}--\eqref{a4}. 
		\end{itemize}
		Now we have checked that $\tld\rho_0$ satisfies all assumptions of Theorem~\ref{thm:to-infty}. Applying the theorem gives  
		\[
		\inf_{a\in\mathbb{R}^2} \|\tld\rho(\cdot,t) - (\tld\rho_0)^*(\cdot-a)\|_{L^1(\mathbb{R}^2)} \leq C_0 \sqrt{A\varep} \quad\text{ for all }t\geq 0,
		\]
		where $C_0$ is a universal constant, and
		\[	\frac{\tld X_2(t)}{\tld X_{02}} \in (1-400 \varep, 1] \quad\text{ for all } t\geq 0.
		\]
		Finally, we replace $\tld \rho(x,t)$ by $\lambda^2 \rho(\lambda x, \lambda^2 t)$ in the above inequalities (again, recall that the $L^1$ norm is preserved under the scaling) to get the desired estimates.
	\end{proof}
	
	Next we give a family of $\rho_0$ that satisfies the assumptions \eqref{a3new}--\eqref{a4new}. This result will be used later in the proof of our main theorem.
	
	\begin{proposition}
		\label{prop_assumption}
		For any $\varep\in(0,\frac{1}{1000})$, assume $p = (p_1,p_2)\in Q$ satisfies $p_1 > 8\varep^{-1}p_2$, and $\lambda>0$ is sufficiently small such that $\lambda < \frac{1}{4}p_2 \varep$. Then any $\rho_0 \in L_+^1(Q)$ with $\int_{Q} \rho_0 dx =1$ supported in $B(p,\lambda)$ satisfies \eqref{a3new}--\eqref{a4new}.
	\end{proposition} 
	
	\begin{proof}
		First, note that since $\lambda < \frac{1}{4}p_2\varep$ and $\supp\rho_0 \subset B(p,\lambda)$, the vertical center of mass $X_{02}=\int \rho_0 x_2 dx$ satisfies
		\[
		X_{02} > p_2-\lambda > p_2\left(1-\frac{\varep}{4}\right) > \frac{p_2}{2} > \varep^{-1}\lambda,
		\]
		and we also have the upper bound $X_{02} < p_2\left(1+\frac{\varep}{4}\right)$. Also, for any $x,y\in\supp\rho_0 \subset B(p, \lambda)$, we have $|x-\bar y|\geq 2p_2 - 2\lambda \geq 2p_2(1 - \frac{1}{4}\varep)$, thus
		\[
		\begin{split}
			E_4[\rho_0] &= \iint_{\mathbb{R}^2\times\mathbb{R}^2} \rho_0(x)\rho_0(y)\log|x-\bar y| dxdy\\
			& \geq \log(2p_2) + \log \left(1 - \frac{1}{4}\varep\right)\\
			&\geq \log(2X_{02}) - \log\left(1 + \frac{1}{4}\varep\right) + \log \left(1 - \frac{1}{4}\varep\right)\\
			& \geq \log(2X_{02}) - \varep, 
		\end{split}
		\]
		where we used that $\varep < \frac{1}{1000}$ in the last step. So we have verified that $\rho_0$ satisfies \eqref{a3new}.
		
		To verify \eqref{a4new}, recall that
		\[
		E_2[\rho_0]+E_3[\rho_0] = \iint_{\mathbb{R}^2\times\mathbb{R}^2} \rho(x)\rho(y) \log\frac{|x-\tld y|}{|x+y|} dxdy \geq \inf_{x,y\in\supp\rho_0} \log\frac{|x-\tld y|}{|x+y|} .
		\]
		For any $x,y\in \supp\rho_0$, using that $\lambda<p_2$ and $\frac{p_2}{p_1}<\frac{\varep}{8}$, we have
		\[
		\frac{|x-\tld y|}{|x+y|} \geq \frac{2(p_1-\lambda)}{2(\sqrt{p_1^2+p_2^2} + \lambda)} \geq \frac{p_1-p_2}{\sqrt{p_1^2+p_2^2} + p_2} \geq \frac{1-\frac{\varep}{8}}{\sqrt{1+(\frac{\varep}{8})^2} + \frac{\varep}{8}} \geq 1-\frac{1}{2}\varep.
		\] 
		This leads to 
		\[
		\inf_{x,y\in\supp\rho_0} \log\frac{|x-\tld y|}{|x+y|} \geq \log\left(1-\frac{1}{2}\varep\right) \geq -\varep,
		\]
		finishing the proof.
	\end{proof}

	\subsection{Orbital stability starting near an arbitrary point}
	\label{subsec_final}
	
	In this subsection, we will consider $\rho_0 = \omega_0 \mathbf{1}_Q$ to be concentrated near an arbitrary point $p \in Q$. To begin with, we recall results on the justification of the point vortex motion, often called as a desingularization problem. Even local in time justification is a highly non-trivial problem, which was first done by Marchioro--Pulvirenti \cite{MaP83} with extensions in \cite{MaP93}. A nice exposition is given in Marchioro--Pulvirenti \cite[Chap. 4]{MaPu}.

	For our application, we can focus on the initial setting of four point vortices in $\mathbb{R}^2$ that are odd-odd, where the initial vortex in $Q$ is located at point $p$ and has strength 1. As the point vortices evolve in time according to the 2D Euler equations, we denote the trajectory of the point vortex in $Q$ at time $t$ by $Z_p(t)$. 
	
	The following theorem\footnote{The statement in \cite[Theorem 2.1]{MaP93} is in fact stronger than what we state here: for a given finite time interval $[0,T]$,  their desingularization result can be done near any point vortex solution (not just in the odd-odd setting as we stated). Also, their $L^\infty$ norm assumption is slightly less restrictive than what we state (they only need $\|\omega_0\|_{L^\infty} \leq A\delta^{-\eta}$ for some $\eta<8/3$, and we fix $\eta=2$ for our application). We only state the version that we need for our application.}
	by Marchioro and Pulvirenti says that for any given \emph{finite time interval} $[0,T]$ and any radius $\lambda$, $\rho(\cdot, t)=\omega(\cdot, t)\mathbf{1}_Q$ will stay concentrated near $Z_p(t)$ if $\rho_0$ is sufficiently concentrated near $p$: 
	
	\begin{proposition}[{{\cite[Theorem 2.1]{MaP93}}}] \label{prop:MP}
		Take any point $p\in Q$,  and $A, \lambda, T>0$. Then  there exists a $\dlt_{0} = \delta_0(p, A, \lambda, T) \in (0,\lambda)$, such that the following holds:
		for all $0<\dlt \le \dlt_0$, if the initial data $\omega_0$ to \eqref{eq:Euler} is odd-odd and $\rho_0 := \omega_0|_Q$ satisfies \begin{equation*}
			\begin{split}
				\int_{\mathbb{R}^2} \rho dx = 1, \quad \supp \rho_0 \subset B(p, \delta), \quad \text{ and }\|\rho_0\|_{L^\infty(\mathbb{R}^2)} \leq A\delta^{-2},
			\end{split}
		\end{equation*} then $\rho(\cdot,t):=\omega(\cdot,t)\mathbf{1}_Q$ satisfies \begin{equation*}
			\begin{split}
				\mathrm{supp}\, \rho(\cdot,t) \subset B(Z_p(t), \lambda) \quad \mbox{for all}\quad t \in [0,T].
			\end{split}
		\end{equation*}

	\end{proposition}

	We are now ready to complete the proof of Theorem \ref{thm:main1}. \begin{proof}[\textbf{\emph{Proof of Theorem \ref{thm:main1}}}]
		The proof is divided into the following steps: Steps 1--4 are mainly devoted to fixing various parameters. The main proof for orbital stability for positive times is done in Steps 5--6. In the final Step 7, we prove orbital stability for negative times.
		
		\medskip
		\noindent\textbf{Step 1.} Fixing a small $\tilde\varep$ depending on $\varep, A$ and $p$. 
		
		For any $\varep, A$ and $p=:(p_1,p_2)$ given in Theorem~\ref{thm:main1}, let us define a sufficiently small $\tilde\varep = \tilde\varep(\varep, A, p)>0$, given by
		\begin{equation}\label{def_tilde_eps}
			\tilde\varep := \min\left\{C_0^{-2} A^{-1}\varepsilon^2, ~\frac{1}{4}c(K_1, 1)A^{-1} \varepsilon^2, ~  \frac{\sqrt{p_1^2+p_2^2}}{2400 p_1p_2} \varep, ~\frac{1}{1000} \right\}.
		\end{equation}
		Here $C_0$ and $c(K_1,1)$ (with $K_1=-\log(\cdot)$) are the universal constants given in Corollary~\ref{cor:rescale} and  Theorem~\ref{thm_yy} respectively. We will work with such $\tilde\varep$ throughout the proof; the motivation for its definition will become clear in Steps 5--6.
		 
		\medskip
		\noindent\textbf{Step 2.} Fixing  a large time $T>0$ and a point $q\in \mathcal{O}_{pv}(p)$.

		For $p \in Q$ given by Theorem~\ref{thm:main1}, we consider the point vortex dynamics $Z_p(t)$ defined in the second paragraph of Section~\ref{subsec_final}. For simplicity, we omit the $p$ dependence and call it $Z(t) := (Z_1(t),Z_2(t))$. Recall that $Z_1(t)$ monotone increases in time, and goes to $+\infty$ as $t\to\infty$; whereas $Z_2(t)$ monotone decreases in time, and goes to 
		\begin{equation}\label{def_x2infty}
			Z_2(\infty) := \frac{p_1 p_2}{\sqrt{p_1^2+p_2^2}}\quad  \text{ as }t\to\infty.
		\end{equation}
		Therefore there exists a sufficiently large $T>0$, such that $q = (q_1, q_2) := Z(T)$ satisfies 
		\begin{equation}
			\label{a_p}q_1 > 8\tilde\varep^{-1} q_2, \quad q_2 \leq \min\left\{ 2Z_2(\infty), ~Z_2(\infty) + \frac{\varep}{4}\right\}
		\end{equation} with $\tilde\varep$ given in Step 1. Note that such $q$ belongs to $\mathcal{O}_{pv}(p)$, and we also have $q_1>p_1$ for free since $Z_1$ is monotone increasing in time.

		\medskip
		\noindent\textbf{Step 3.} Fixing a small radius $\lambda>0$ around $q$.
		
		Let us pick a sufficiently small $\lambda>0$ such that it satisfies\footnote{The reason for the second and third argument in the min function will only be clear in Step 6; we will not use it for a while.}
		\begin{equation}
			\label{def_lambda} \lambda < \min\left\{\frac{1}{4}q_2 \tilde\varep, ~\frac{1}{4}\varep, ~ Z_2(\infty)\right\},
		\end{equation}
		where $\tilde\varep$ and $q_2$ are fixed in Step 1 and 2 respectively, and $Z_2(\infty)$ is defined in \eqref{def_x2infty}.
		Note that now $\tilde \varep, q$ and $\lambda$ satisfy the assumptions of Proposition~\ref{prop_assumption}, which we will apply later.
		
		Next, we will further reduce $\lambda$, such that it satisfies
		\begin{equation}\label{K234_diff}
			\sup_{t\in[0,T]}\sup_{x,y\in B(Z(t), \lambda)} |K_{234}(x,y) - K_{234}(Z(t),Z(t))| < \frac{1}{4}\tilde\varep, 
		\end{equation}
		where $K_{234}$ is defined as in \eqref{K234_def}. This is doable since $K_{234}$ is continuous in $Q\times Q$, and $\{(Z(t),Z(t)):t\in[0,T]\}$ is a compact set in $Q\times Q$.

		\medskip
		\noindent\textbf{Step 4.} Fixing a small radius $\delta_0$ around $p$.

		Let us apply Proposition~\ref{prop:MP} to the point $p$ and constant $A$ in Theorem~\ref{thm:main1}, the radius $\lambda$ in Step~3, and the time $T$ in Step 2. It gives a $\delta_0 \in (0,\lambda)$ such that satisfies the following property: for any $\delta \in (0,\delta_0)$ and $\rho_0$ with mass 1 that satisfies
		$
		\supp\rho_0 \subset B(p,\delta)$ and $ \|\rho_0\|_{L^\infty} \leq A \delta^{-2}$
		(note that these form a subset of our assumption \eqref{a_main}),
		we have 
		\begin{equation}\label{supp_T}
			\supp \rho(\cdot,t) \subset B(Z(t),\lambda) \quad\text{ for all }t\in[0,T].
		\end{equation}
		
		In addition, we claim that we can further reduce $\delta_0$, such that for any $\delta\in(0,\delta_0)$, the assumption \eqref{a_main} implies that 
		\begin{equation}\label{e1_diff}
			|E_1[\rho_0]-E_1[\rho_0^*]|\leq \frac{\tilde\varep}{2}.
		\end{equation}
		To see why this holds, let us bound $|E_1[\rho_0]-E_1[\rho_0^*]|$ as follows under the assumption \eqref{a_main}:
		\[
		\begin{split}
			|E_1[\rho_0]-E_1[\rho_0^*]| &= |E_1[\rho_0]-E_1[\rho_0^*(\cdot-p)]| \\
			&= \left|\iint_{\mathbb{R}^2\times\mathbb{R}^2} (\rho_0(x) - \rho_0^*(x-p))  (\rho_0(y) + \rho_0^*(y-p)) K_1(x-y) dxdy\right| \\
			&\leq \left|\iint_{\mathbb{R}^2\times\mathbb{R}^2} (\rho_0(x) - \rho_0^*(x-p)) \mathbf{1}_{x,y\in B(p,\delta)} 2A\delta^{-2} K_1(x-y) dxdy\right|\\
			&\leq \left|\iint_{\mathbb{R}^2\times\mathbb{R}^2} (\rho_0(x) - \rho_0^*(x-p))   \mathbf{1}_{z\in B(0,2\delta)} 2A\delta^{-2} K_1(z) dx dz\right| \quad(z:=x-y)\\
			&\leq 2A\delta^{-2} \|\rho_0 - \rho_0^*(\cdot-p)\|_{L^1} \left|\int_{B(0,2\delta)} \log |z| dz\right|\\
			&\leq C\delta |\log \delta|,
		\end{split}
		\]
		where the first inequality follows from the assumption $\|\rho_0\|_{L^\infty}\leq A\delta^{-2}$ (thus $|\rho_0+\rho_0^*(\cdot-p)|\leq 2A\delta^{-2}$) and the assumption $\supp\rho_0\subset B(p,\delta)$ 
		(which leads to $\supp\rho_0^* \subset B(0,\delta)$, thus $\supp\rho_0^*(\cdot-p) \subset B(p,\delta)$). The final inequality follows from the assumption $\|\rho_0 - \rho_0^*(\cdot-p)\|_{L^1}$ in \eqref{a_main}. Since the function $s|\log s|$ is increasing in $(0,e^{-1})$, by further reducing $\delta_0$ such that $\delta_0 \in (0,e^{-1})$ and $\delta_0|\log\delta_0|\leq \tilde\varep/2$, we have $\delta|\log\delta| \leq \delta_0|\log\delta_0| \leq \tilde\varep/2$ for all $\delta \in (0,\delta_0)$. This finishes the proof of the claim \eqref{e1_diff}.
		
		\medskip
		\noindent\textbf{Step 5.} Applying Corollary~\ref{cor:rescale} to $\rho(\cdot,T)$.

		In this step, assuming  $\rho_0$ satisfies \eqref{a_main} for some $\delta \in (0,\delta_0)$, we aim to apply  Corollary~\ref{cor:rescale} to $\rho(\cdot,T)$, with parameters $\tilde\varep, A$ and $\delta$. Let us check that all assumptions of Corollary~\ref{cor:rescale} are satisfied:
		
		\begin{itemize}
			\item Verifying \eqref{a1new}: Since $\rho_0$ satisfies \eqref{a_main}, and $\rho(\cdot,T)$ has the same mass, $L^\infty$-norm and support size as $\rho_0$, it is straightforward to verify that \eqref{a1new} holds (with $\lambda$ replaced by $\delta$), i.e.
			\[
			\int_{\mathbb{R}^2} \rho(x,T) dx = 1, \quad 0\leq \rho(\cdot,T) \leq A \delta^{-2} \quad  \text{ and }\quad |\supp \rho(\cdot,T)| \leq \pi \delta^2.
			\]
			
			\medskip
			\item Verifying \eqref{a3new}--\eqref{a4new}:
			
			Recall that in Step 3,  $q$ and $\lambda$ are chosen such that they satisfy the assumptions of Proposition~\ref{prop_assumption} with parameter $\tilde\varep$. In Step 4, we showed that when $\rho_0$ satisfies \eqref{a_main}, $\supp\rho(\cdot,T)\subset B(q,\lambda)$ (which follows from \eqref{supp_T}). We can then apply Proposition~\ref{prop_assumption} to conclude that $\rho(\cdot,T)$ satisfies \eqref{a3new}--\eqref{a4new} (with parameter $\varep$ replaced by $\tilde\varep$). More precisely, we have
			\begin{equation}\label{a3new2}
				X_2(T) := \int_{\mathbb{R}^2}\rho(x,T) x_2 dx > \tilde\varep^{-1}\lambda > \tilde\varep^{-1}\delta \quad\text{and}\quad E_4[\rho(T)] - \log(2X_2(T)) \geq - \tilde\varep;
			\end{equation}
			\[
			E_2[\rho(T)] + E_3[\rho(T)] \geq- \tilde\varep.
			\]
			Note that in the last step of the first inequality of \eqref{a3new2}, we also used the fact that $\lambda \geq \delta_0 > \delta$ (where $\lambda > \delta_0$ is due to Step 4, and $\delta\in(0,\delta_0)$ comes from our assumption in Theorem~\ref{thm:main1}).
			
			\medskip
			\item Verifying \eqref{a2new}:
			
			Finally, we aim to show that for any $t\in [0,T]$, $\rho(\cdot,t)$ satisfies\footnote{To show \eqref{a2new}, we only need to take $t=T$. But the estimate for $t\in(0,T)$ will be useful in the next step.}
			\begin{equation}\label{e1_goal}
				E_1[\rho(t)] - \underbrace{E_1[\rho(t)^*]}_{=E_1[\rho_0^*]} \geq -\tilde\varep.
			\end{equation}
			Note that we have $\rho(\cdot,t)^*=\rho_0^*$, since $\rho(\cdot,t)$ has the same distribution function for all time (as it is transported by a divergence free velocity field).
			Recall that we already have \eqref{e1_diff} under the assumption \eqref{a_main} and our choice of $\delta_0$. Thus to show \eqref{e1_goal}, it suffices to show that
			\[
			|E_1[\rho_0] - E_1[\rho(t)]| \leq \frac{\tilde\varep}{2} \quad\text{ for all }t\in[0,T].
			\]
			Since $\tld E[\rho(t)]=\tld E[\rho_0]$ is conserved, this desired inequality is equivalent to
			\begin{equation}\label{goal_234}
				|E_{234}[\rho_0] - E_{234}[\rho(t)]| \leq \frac{\tilde\varep}{2} \quad\text{ for all }t\in[0,T].
			\end{equation}
			To show this, recall that for any $t\in [0,T]$, since $\rho(\cdot,t)$ have mass 1 and $\supp \rho(\cdot,t) \subset B(Z(t),\lambda)$ (due to \eqref{supp_T}), we have
			\[
			\begin{split}
				|E_{234}[\rho(t)] - K_{234}(Z(t), Z(t))| &= \left| \iint_{\mathbb{R}^2\times\mathbb{R}^2} (K_{234}(x,y)-K_{234}(Z(t), Z(t))) \rho(x,t)\rho(y,t)dxdy\right|\\
				&\leq \sup_{x,y\in \supp \rho(\cdot,t)} \left|K_{234}(x,y)-K_{234}(Z(t), Z(t))\right|\\ 
				&\leq \frac{\tilde\varep}{4},
			\end{split}
			\]
			where the last step follows from \eqref{K234_diff} and $\supp \rho(\cdot,t) \subset B(Z(t),\lambda)$. Applying this inequality at time $0$ and $t$ respectively and combining it with the fact that $K_{234}(Z(0),Z(0)) = K_{234}(Z(t),Z(t))$ (which follows from \eqref{K234_eq}) gives us \eqref{goal_234}, which yields \eqref{e1_goal}.
			
		\end{itemize}
		
		Finally, we are ready to apply Corollary~\ref{cor:rescale} to $\rho(\cdot,T)$ with parameters $\tilde\varep, A$ and $\delta$. It leads to
		\begin{equation}\label{conclusion11}
			\inf_{a\in\mathbb{R}^2} \|\rho(\cdot,t) - \rho_0^*(\cdot-a)\|_{L^1(\mathbb{R}^2)} \leq C_0 \sqrt{A\tilde\varep} \leq \varep \quad\text{ for all } t\geq T,
		\end{equation}
		and
		\begin{equation}\label{conclusion22}
			\frac{X_2(t)}{X_2(T)} \in (1- 400\tilde\varep, 1] \subset \left(1-  \frac{\sqrt{p_1^2+p_2^2}}{6p_1p_2}\varep, 1\right] \quad\text{ for all } t\geq T.
		\end{equation}
		where the last step of the two inequalities follows from our choice of $\tilde\varep$ in Step 1. Note that \eqref{conclusion11} implies \eqref{thm_goal1} for all $t\geq T$.
		
		\medskip
		\item\textbf{Step 6. Orbital stability for all $t\geq 0$.}
		
		In this step, we aim to prove \eqref{thm_goal1} for $0\leq t\leq T$ (recall that the $t\geq T$ case is already done in Step 5 due to \eqref{conclusion11}), and \eqref{thm_goal2} for $t\geq 0$.

		Defining $\mu(t) := \delta^2 \rho(\delta x, t)$, it is easy to check that
		\[
		E_1[\mu(t)] = E_1[\rho(t)] + \log \delta,
		\]
		and combining this with \eqref{e1_goal} gives
		\[
		E_1[\mu(0)^*] - E_1[\mu(t)] \leq \tilde\varep \quad\text{ for all }t\in [0,T].
		\]
		Note that for any $t\in[0,T]$, $\mu(\cdot,t)$ has mass 1, and satisfies $0\leq \mu(\cdot,t)\leq A$ and $\supp\mu(\cdot,0)^* \subset B(0,1)$, where the two inequalities follows from \eqref{a_main} and the scaling property. We can then apply Theorem~\ref{thm_yy} to obtain
		\[
		\tilde\varep \geq E_1[\mu(0)^*]- E_1[\mu(t)]  \geq \frac{1}{4} c(K_1,1) A^{-1} \inf_{a\in\mathbb{R}^2}\|\mu(\cdot,t)-\mu^*(\cdot-a, 0)\|_{L^1}^2.
		\]
		Since $\mu$ and $\rho$ have the same $L^1$ norm, this leads to
		\[
		\inf_{a\in\mathbb{R}^2}\|\rho(\cdot,t)-\rho_0^*(\cdot-a)\|_{L^1} = \inf_{a\in\mathbb{R}^2}\|\mu(\cdot,t)-\mu^*(\cdot-a, 0)\|_{L^1} \leq \sqrt{4c(K_1,1)^{-1} A \tilde \varep} \leq \varep,
		\]
		where the last step follows from our choice of $\tilde\varep$ in \eqref{def_tilde_eps}. This finishes the proof for \eqref{thm_goal1} for $t\geq 0$.
		
		\medskip
		It remains to show \eqref{thm_goal2} for all $t\geq 0$. Again, we split the proof into the cases $t\in[0,T]$ and $t>T$. The case $t\in[0,T]$ is easy: in Step 3, we had chosen $\lambda$ such that $\lambda < \frac{1}{4}\varep$. By \eqref{supp_T}, we have
		\begin{equation}\label{temp_diffx}
			|X(t)-Z(t)|\leq \lambda < \frac{\varep}{4}\quad \text { for all }t\in [0,T],
		\end{equation}
		and since $Z(t)\in \mathcal{O}_{pv}(p)$, clearly this leads to \eqref{thm_goal2} for all $t\in[0,T]$.
		
		\medskip
		We now move on to the case $t\geq T$. In this case, we first claim that 
		\begin{equation}\label{temp_diff2}
			0\leq X_2(T)-X_2(t) \leq \frac{\varep}{2} \quad \text{ for all }t\geq T.
		\end{equation}
		Here the lower bound directly follows from Lemma~\ref{lemma_center}. To show the upper bound, we have
		\[
		X_2(T)-X_2(t) \leq \frac{\varep}{6}\frac{X_2(T)}{Z_2(\infty)} \leq \frac{\varep}{6}\left(\frac{Z_2(T)}{Z_2(\infty)} + \frac{X_2(T)-Z_2(T)}{Z_2(\infty)}\right) \leq \frac{\varep}{6}\left(2+\frac{\lambda}{Z_2(\infty)}\right)\leq \frac{\varep}{2},
		\]
		where we applied \eqref{conclusion22} and \eqref{def_x2infty} in the first inequality,  \eqref{a_p} and \eqref{temp_diffx} (for time $T$) in the third inequality, and \eqref{def_lambda} in the last inequality. This proves the claim \eqref{temp_diff2}.
		
		For $t\geq T$, note that the monotonicity result in Lemma~\ref{lemma_center} together with \eqref{temp_diffx} (for time $T$) yield 
		$
		X_1(t)\geq X_1(T) \geq Z_1(T) - \frac{\varep}{4}.
		$ Below we discuss two subcases:
		\begin{itemize}
			\item 
			Case 1. If $t\geq T$ is such that $X_1(t) \in (Z_1(T) - \frac{\varep}{4}, Z_1(T))$, then we directly have
			\[
			\begin{split}
				|X(t)-Z(T)| &\leq |X_1(t)-Z_1(T)| + |X_2(t)- X_2(T)| + | X_2(T)-Z_2(T)| \\
				&\leq \frac{\varep}{4}+  \frac{\varep}{2} +  \frac{\varep}{4} = \varep,
			\end{split}
			\]
			where the second inequality follows from the assumption for Case 1, \eqref{temp_diff2}, and \eqref{temp_diffx} (for time $T$). This directly implies \eqref{thm_goal2} since $Z(T)\in \mathcal{O}_{pv}(p)$.
			
			\item 
			Case 2. Otherwise, we have $X_1(t) \geq Z_1(T)$. Since $Z_1$ is increasing in time and it goes to infinity as $t\to\infty$, there exists some $\tilde T \geq T$ (where $\tilde T$ depends on $t$), such that 
			\[
			X_1(t) = Z_1(\tilde T).
			\]
			We then have
			\[
			\begin{split}
				|X(t) - Z(\tilde T)| &= |X_2(t) - Z_2(\tilde T)|\\
				&\leq |X_2(t)-X_2(T)| + | X_2(T) - Z_2( T)| + |Z_2(T) - Z_2(\tilde T)|\\
				&\leq \frac{\varep}{2} + \frac{\varep}{4} + \frac{\varep}{4} \leq \varep
			\end{split}
			\]
			where in the second inequality we used \eqref{temp_diff2}, \eqref{temp_diffx} (for time $T$), and \eqref{a_p} (also note that $Z_2(T) =q_2 \leq Z(\infty)+\frac{\varep}{4}$ implies $|Z_2(T) - Z_2(\tilde T)| \leq \frac{\varep}{4}$, since $Z_2(\infty) < Z_2(\tilde T) \leq Z_2(T)$). Again, the above inequality directly implies \eqref{thm_goal2} since $Z(\tilde T)\in \mathcal{O}_{pv}(p)$.
		\end{itemize}
		Combining the above cases, we have obtained \eqref{thm_goal2} for all $t\geq 0$. 
		
		\medskip
		\item\textbf{Step 7. Orbital stability for all $t\leq 0$.}
		
		Finally, it remains to show \eqref{thm_goal1} and \eqref{thm_goal2} for $t\leq 0$, which is a simple consequence of the time-reversibility of the Euler equation. Namely, introducing $\eta_0(x_1,x_2) = \rho_0(x_2,x_1)$ and consider the solution with odd-odd initial data whose restriction in $Q$ being $\eta_0$ and $\rho_0$ respectively, the two solutions satisfy
		\[
		\eta(x_1,x_2,t) = \rho(x_2,x_1,-t) \quad\text{ for all }x_1,x_2\in Q, t\in\mathbb{R}.
		\]
		Therefore, to obtain orbital stability for $\rho(\cdot,t)$ for time $t\leq 0$, it suffices to obtain orbital stability for $\eta(\cdot,t)$ for $t\geq 0$. 
		Also, note that if $\rho_0$ satisfies \eqref{a_main}, then $\eta_0$ also satisfies \eqref{a_main} with the point $p=(p_1,p_2)$ replaced by $q=(p_2,p_1)$. 
		
		In Steps 1--6, we have already proved that Theorem~\ref{thm:main1} holds for $t\geq 0$. Therefore, when the point $p$ is replaced by $q$, there exists a new $\tilde\delta_0 = \tilde\delta_0(\varep, A, q)$,  such that for all $\delta \in (0,\tilde\delta_0)$, if $\eta_0$ satisfies \eqref{a_main} (with point $p$ replaced by $q$), then \eqref{thm_goal1} and \eqref{thm_goal2} holds for $\eta(\cdot,t)$ for all $t\geq 0$.  This implies that \eqref{thm_goal1} and \eqref{thm_goal2} holds for $\rho(\cdot,t)$ for all $t\leq 0$, where we also use the fact that the set $\mathcal{O}_{pv}(p)$ is symmetric about $x_1=x_2$.
		
		Finally, by setting $\delta_0$ as the minimum of $\delta_0$ in Step 4 and $\tilde\delta_0$, we conclude the proof of Theorem~\ref{thm:main1}.
	\end{proof}

	\bibliographystyle{plain}

\end{document}